\documentclass[dvips]{amsart}

\title[The relative modular object]{The relative modular object and Frobenius extensions of finite Hopf algebras}
\author{Kenichi Shimizu}
\date{}
\email{kshimizu@shibaura-it.ac.jp}
\address{Department of Mathematical Sciences, Shibaura Institute of Technology, 307
Fukasaku, Minuma-ku, Saitama-shi, Saitama 337-8570, Japan.}

\usepackage[numbers,sort&compress]{natbib}

\usepackage{amsmath,amssymb,amscd,bbm,stmaryrd}
\usepackage[all]{xy}
\usepackage{url}
\usepackage{subcaption}

\usepackage{graphicx}

\usepackage{amsthm}
\numberwithin{equation}{section}

\newtheorem{counter}{}[section]

\theoremstyle{definition}
\newtheorem{definition}         [counter]{Definition}
\newtheorem{notation}           [counter]{Notation}

\theoremstyle{plain}
\newtheorem{lemma}              [counter]{Lemma}

\newtheorem{proposition}        [counter]{Proposition}
\newtheorem{theorem}            [counter]{Theorem}
\newtheorem{corollary}          [counter]{Corollary}

\newtheorem*{theorem*}          {Theorem}
\theoremstyle{remark}
\newtheorem{remark}             [counter]{Remark}
\newtheorem{example}            [counter]{Example}

\newcommand{\id}{\mathrm{id}}
\newcommand{\eval}{{\rm ev}}
\newcommand{\coev}{{\rm coev}}
\newcommand{\op}{\mathsf{op}}
\newcommand{\rev}{\mathsf{rev}}

\newcommand{\env}{\mathsf{env}}
\newcommand{\unitobj}{\mathbbm{1}}
\newcommand{\barotimes}{\mathbin{\tilde{\otimes}}}

\newcommand{\Hom}{\mathrm{Hom}}
\newcommand{\End}{\mathrm{End}}

\newcommand{\iHom}{\underline{\mathrm{Hom}}}

\newcommand{\NAT}{\textsc{Nat}}

\newcommand{\REX}{\textsc{Rex}}
\newcommand{\LEX}{\textsc{Lex}}

\begin{document}

\begin{abstract}
  For a certain kind of tensor functor $F: \mathcal{C} \to \mathcal{D}$, we define the relative modular object $\chi_F \in \mathcal{D}$ as the ``difference'' between a left adjoint and a right adjoint of $F$. Our main result claims that, if $\mathcal{C}$ and $\mathcal{D}$ are finite tensor categories, then $\chi_F$ can be written in terms of a categorical analogue of the modular function on a Hopf algebra. Applying this result to the restriction functor associated to an extension $A/B$ of finite-dimensional Hopf algebras, we recover the result of Fischman, Montgomery and Schneider on the Frobenius type property of $A/B$. We also apply our results to obtain a ``braided'' version and a ``bosonization'' version of the result of Fischman et al.
\end{abstract}

\maketitle

\section{Introduction}

Frobenius-type properties of extensions of Hopf algebras and of related algebras have been studied extensively; see, {\it e.g.}, \cite{MR1194305,MR1468830,MR1401518,MR1690111,MR1865524,MR1933711}.
Fischman, Montgomery and Schneider \cite{MR1401518} showed that the Frobenius property of an extension of finite-dimensional Hopf algebras is controlled by their modular functions. The aim of this paper is to formulate and prove a generalization of their result in the setting of finite tensor categories, a class of tensor categories including the representation category of a finite-dimensional Hopf algebra.

To explain our results in more detail, we briefly recall the above-mentioned result of Fischman et al. Recall that, for a finite-dimensional Hopf algebra $H$ over a field $k$, the (right) {\em modular function} $\alpha_H: H \to k$ (also called the distinguished grouplike element in literature) is defined by
\begin{equation}
  \label{eq:mod-func}
  h \cdot \Lambda = \alpha_H(h) \Lambda \quad (h \in H),
\end{equation}
where $\Lambda \in H$ is a non-zero right integral. For an extension $A/B$ of finite-dimensional Hopf algebras over $k$ (meaning that $A$ is such a Hopf algebra and $B$ is a Hopf subalgebra of $A$), the {\em relative modular function} $\chi = \chi_{A/B}$ and the {\em relative Nakayama automorphism} $\beta = \beta_{A/B}$ of $A/B$ are defined respectively by
\begin{equation}
  \label{eq:rel-mod-func}
  \chi(b) = \alpha_A(b_{(1)}) \alpha_B(S(b_{(2)}))
  \text{\quad and \quad}
  \beta(b) = \chi(b_{(1)}) b_{(2)}
\end{equation}
for $b \in B$, where $S$ is the antipode of $B$ and $\Delta(b) = b_{(1)} \otimes b_{(2)}$ is the comultiplication of $b$ in the Sweedler notation \cite[Definition 1.6]{MR1401518}. They showed that $A/B$ is a {\em $\beta$-Frobenius extension}, {\it i.e.}, $A$ is finitely-generated and projective as a right $B$-module (by the Nichols-Zoeller theorem) and there is an isomorphism
\begin{equation}
  \label{eq:intro-FMS}
  {}_B A_A \cong {}_{\beta} \Hom_B(A_B, B_B)
\end{equation}
of $B$-$A$-bimodules \cite[Theorem 1.7]{MR1401518}, where ${}_{\beta}(-)$ means the left $B$-module obtained by twisting the action of $B$ by $\beta$.

To formulate this result in a category-theoretical setting, we consider the restriction functor $\mathrm{res}^A_B: \mbox{\rm mod-}A \to \mbox{\rm mod-}B$ between the categories of right modules. By the standard argument and the Nichols-Zoeller theorem, the functors
\begin{equation}
  \label{eq:intro-L-R-1}
  L := (-) \otimes_B A
  \text{\quad and \quad}
  R := \Hom_B(A, -) \cong (-) \otimes_B \Hom_B(A, B)
\end{equation}
are a left adjoint and a right adjoint of $\mathrm{res}^A_B$, respectively. Hence \eqref{eq:intro-FMS} says that the relative modular function measures the ``difference'' between $L$ and $R$. Here we remark that $\mathrm{res}^A_B$ is in fact a tensor functor. Thus, we are led to the problem of studying the ``difference'' between a left adjoint and a right adjoint of a tensor functor.

Based on this observation, we consider a tensor functor $F: \mathcal{C} \to \mathcal{D}$ having a left adjoint $L$ and a right adjoint $R$. It turns out that, under certain assumptions, there exists a unique (up to isomorphism) object $\chi_F \in \mathcal{D}$ such that $L \cong R(\chi_F \otimes -)$. Our main result is that if $\mathcal{C}$ and $\mathcal{D}$ are finite tensor categories, then the object $\chi_F$ is expressed in terms of the category-theoretical analogue of the modular function introduced by Etingof, Nikshych and Ostrik \cite{MR2097289}. Applying this result to $\mathrm{res}^A_B$, we recover the above-mentioned result of Fischman et al. We also apply our results to obtain some variants of their result.

Now we explain the organization of this paper. Throughout, we denote the base field by $k$. In Section~\ref{sec:prelim}, we collect from \cite{MR1712872,MR0450361,MR0498792,MR0498793,MR2119143,MR3242743} basic results on tensor categories and their module categories. We note that, unlike \cite{MR2119143}, the base field $k$ is arbitrary and the unit object of a finite tensor category is not assumed to be simple; see Definition \ref{def:FTC} for our setting.

In Section~\ref{sec:mod-obj}, we introduce the {\em modular object} $\alpha_{\mathcal{C}} \in \mathcal{C}$ of a finite tensor category $\mathcal{C}$ and study its properties. After a brief discussion on ends, coends and the Deligne tensor product, we introduce an algebra $A$ in $\mathcal{C} \boxtimes \mathcal{C}^{\rev}$. If $\mathcal{M}$ and $\mathcal{N}$ are finite left $\mathcal{C}$-module categories (in the sense of Definition~\ref{def:FTC-mod}), then $\mathcal{C} \boxtimes \mathcal{C}^{\rev}$ acts on the category $\REX(\mathcal{M}, \mathcal{N})$ of $k$-linear right exact functors from $\mathcal{M}$ to $\mathcal{N}$, and hence the category of $A$-modules in $\REX(\mathcal{M}, \mathcal{N})$ is defined. A key observation is that an $A$-module in $\REX(\mathcal{M}, \mathcal{N})$ is precisely a $\mathcal{C}$-module functor. Based on this observation, we define the modular object $\alpha_{\mathcal{C}} \in \mathcal{C}$ in an abstract way (Definition~\ref{def:mod-obj}).
It turns out that $\alpha_{\mathcal{C}}$ is isomorphic to the dual of the distinguished invertible object $D \in \mathcal{C}$ introduced in \cite{MR2097289} whenever the definition of $D$ makes sense (Proposition~\ref{prop:mod-obj-D}). Hence, if $\mathcal{C} = \mbox{mod-}H$ for some finite-dimensional Hopf algebra $H$, then $\alpha_{\mathcal{C}}$ is the $H$-module corresponding to the modular function $\alpha_H$ defined by~\eqref{eq:mod-func}.
We also prove a category-theoretical analogue of Radford's formula for the fourth power of the antipode of a finite-dimensional Hopf algebra, which has been proved in \cite{MR2097289} under certain mild assumptions.

In Section~\ref{sec:rel-mod-obj}, we consider a tensor functor $F: \mathcal{C} \to \mathcal{D}$ between tensor categories (in the sense of \S\ref{subsec:ten-fun}) having a left adjoint $L$ and a right adjoint $R$. We show that $L$ has a left adjoint, if and only if $R$ has a right adjoint, if and only if there is an object $\chi_F \in \mathcal{D}$ such that $R \cong L(- \otimes \chi_F)$. It turns out that such an object $\chi_F$ is unique up to isomorphism. Thus we call $\chi_F$ the {\em relative modular object} of $F$. We prove that $\chi_F$ is invertible and there are also natural isomorphisms
\begin{equation*}
  R \cong L(\chi_F \otimes -)
  \quad \text{and} \quad
  R(\chi_F^* \otimes -) \cong L \cong R(- \otimes \chi_F^*).
\end{equation*}
We shall remark that these results may be an instance of a quite general principle in the theory of monoidal categories. Indeed, similar results are obtained in some different settings in \cite{2014arXiv1411.2236B,2015arXiv150101999B}. In any case, the above results are not sufficient as a generalization of the result of Fischman et al.; their result describes the relation between $L$ and $R$ by the relative modular function $\chi_{A/B}$, while ours do not give any information about $\chi_F$. Our main result in Section~\ref{sec:rel-mod-obj} is, in fact, the following formula for the relative modular object:

\begin{theorem*}[Theorem~\ref{thm:mod-obj-3}]
  Let $F: \mathcal{C} \to \mathcal{D}$ be a tensor functor between finite tensor categories admitting the relative modular object. Then we have
  \begin{equation*}
    F(\alpha_{\mathcal{C}}) \otimes \alpha_{\mathcal{D}}^*
    \cong \chi_F \cong \alpha_{\mathcal{D}}^* \otimes F(\alpha_{\mathcal{C}}).
  \end{equation*}
\end{theorem*}

From this theorem, we see that if $F$ is the restriction functor $\mathrm{res}^A_B$ as above, then $\chi_F$ is the one-dimensional $H$-module corresponding to the relative modular function $\chi_{A/B}$. We will explain in detail how to derive \cite[Theorem 1.7]{MR1401518} from our results; see \S\ref{subsec:rel-mod-obj-formula}.

To apply the above theorem in concrete cases, we need an explicit expression of the modular object. Thus, in Section~\ref{sec:braided-Hopf}, we determine the modular object of the category $\mathcal{V}_H$ of right modules over a Hopf algebra $H$ in a braided finite tensor category $\mathcal{V}$. The modular function $\alpha_H: H \to \unitobj$ is defined by a similar formula as~\eqref{eq:mod-func}, however, the right $H$-module corresponding to $\alpha_H$ is not the modular object of $\mathcal{V}_H$ in general. We express the modular object of $\mathcal{V}_H$ by the modular function of $H$, the modular object of $\mathcal{V}$ and the ``object of integrals'' $\mathrm{Int}(H) \in \mathcal{V}$ in the sense of \cite{MR1759389} (Theorem~\ref{thm:mod-obj-br-Hopf}). As an application, we obtain the following ``braided'' version of \cite[Corollary 1.8]{MR1401518}:

\begin{theorem*}[Theorem~\ref{thm:FMS-braided}]
  Let $\mathcal{V}$ be a braided finite tensor category, and let $A/B$ be an extension of Hopf algebras in $\mathcal{V}$. Then the following assertions are equivalent:
  \begin{enumerate}
  \item The restriction functor $\mathrm{res}^A_B: \mathcal{V}_A \to \mathcal{V}_B$ is a Frobenius functor.
  \item $\mathrm{Int}(A) \cong \mathrm{Int}(B)$ and $\alpha_A \circ i = \alpha_B$, where $i: B \to A$ is the inclusion.
  \end{enumerate}
\end{theorem*}

In Section~\ref{sec:boson}, we study the modular object of the tensor category arising from a Hopf algebra in the monoidal center. Given a Hopf algebra $\mathbf{B}$ in the monoidal center $\mathcal{Z}(\mathcal{C})$ of a finite tensor category $\mathcal{C}$, we obtain the finite tensor category ${}_{\mathbf{B}}\mathcal{C}$ of the category of $\mathbf{B}$-modules in $\mathcal{C}$. This category is a category-theoretical counterpart of the Radford-Majid bosonization. By using the results of Section~\ref{sec:braided-Hopf}, we give a formula for the modular object of ${}_{\mathbf{B}}\mathcal{C}$ (Theorem~\ref{thm:mod-obj-boson}). This formula yields another generalization of the result of Fischman et al., which tells when the restriction functor $\mathrm{res}^{\mathbf{A}}_{\mathbf{B}}: {}_{\mathbf{A}}\mathcal{C} \to {}_{\mathbf{B}}\mathcal{C}$ associated to an extension $\mathbf{A}/\mathbf{B}$ of Hopf algebras in $\mathcal{Z}(\mathcal{C})$ is Frobenius (Corollary~\ref{cor:boson-2}).

\section*{Acknowledgments}

The author was supported by Grant-in-Aid for JSPS Fellows (24$\cdot$3606) when part of this work was done.
The author is currently supported by JSPS KAKENHI Grant Number JP16K17568.

\section{Preliminaries}
\label{sec:prelim}

\subsection{Monoidal categories}

A {\em monoidal category} \cite[VII.1]{MR1712872} is a category $\mathcal{C}$ endowed with a functor $\otimes: \mathcal{C} \times \mathcal{C} \to \mathcal{C}$ (called the tensor product), an object $\unitobj \in \mathcal{C}$ (called the unit object) and natural isomorphisms
\begin{equation*}
  (X \otimes Y) \otimes Z \cong X \otimes (Y \otimes Z) \quad \text{and} \quad
  \unitobj \otimes X \cong X \cong X \otimes \unitobj
\end{equation*}
obeying the pentagon and the triangle axiom. If these natural isomorphisms are the identity, then $\mathcal{C}$ is said to be strict. By the Mac Lane coherence theorem, we may assume that all monoidal categories are strict. Given a monoidal category $\mathcal{C}$, we denote by $\mathcal{C}^{\rev}$ the same category but with the reversed tensor product given by $X \otimes^{\rev} Y = Y \otimes X$.

Our convention for dual objects follows Kassel's book \cite{MR1321145}: Let $L$ and $R$ be objects of $\mathcal{C}$, and let $\varepsilon: L \otimes R \to \unitobj$ and $\delta: \unitobj \to R \otimes L$ be morphisms in $\mathcal{C}$. We say that $(L, \varepsilon, \eta)$ is a {\em left dual object} of $R$ and $(R, \varepsilon, \eta)$ is a {\em right dual object} of $L$ if $\varepsilon$ and $\eta$ satisfy the zig-zag equations. We say that $\mathcal{C}$ is {\em rigid} if every object of $\mathcal{C}$ has a left dual object and a right dual object. If this is the case, we denote by $(V^*, \eval, \coev)$ the (fixed) left dual object of $V \in \mathcal{C}$. The assignment $V \mapsto V^*$ extends to an equivalence $(-)^*: \mathcal{C} \to \mathcal{C}^{\op,\rev}$ of monoidal categories, which we call the {\em left duality functor}. A quasi-inverse of $(-)^*$, denoted by ${}^*(-)$, is given by taking a right dual object. For simplicity, we assume that $(-)^*$ and ${}^*(-)$ are strict monoidal and mutually inverse to each other.

\subsection{Modules over a monoidal category}

Let $\mathcal{C}$ be a monoidal category. A {\em left $\mathcal{C}$-module category} is a category $\mathcal{M}$ endowed with a functor $\ogreaterthan: \mathcal{C} \times \mathcal{M} \to \mathcal{M}$ (called the action) and natural isomorphisms
\begin{equation*}
  a_{X,Y,M}: (X \otimes Y) \ogreaterthan M \to X \ogreaterthan (Y \ogreaterthan M)
  \quad \text{and} \quad
  \ell_M: \unitobj \ogreaterthan M \to M
\end{equation*}
obeying certain axioms similar to those for a monoidal category. If $\mathcal{M}$ and $\mathcal{N}$ are left $\mathcal{C}$-module categories, then a (left) {\em lax $\mathcal{C}$-module functor} from $\mathcal{M}$ to $\mathcal{N}$ is a functor $F: \mathcal{M} \to \mathcal{N}$ endowed with a natural transformation
\begin{equation*}
  \xi^{F}_{X,M}: X \ogreaterthan F(M) \to F(X \ogreaterthan M)
  \quad (X \in \mathcal{C}, M \in \mathcal{M})
\end{equation*}
satisfying certain coherent conditions. If the natural transformation $\xi^F$ is invertible, then $F$ is said to be a {\em strong $\mathcal{C}$-module functor}.
Note that $\mathcal{M}^{\op}$ is naturally a left $\mathcal{C}^{\op}$-module category.
A {\em colax $\mathcal{C}$-module functor} from $\mathcal{M}$ to $\mathcal{N}$ can be defined to be a lax $\mathcal{C}^{\op}$-module functor from $\mathcal{M}^{\op}$ to $\mathcal{N}^{\op}$.
See \cite{MR1976459,2014arXiv1406.4204D,MR3242743} for the precise definitions of these and related notions.
For reader's convenience, we note the following two well-known lemmas:

\begin{lemma}[{\cite[Lemma 2.10]{2014arXiv1406.4204D}}]
  \label{lem:C-mod-func-rigid}
  If $\mathcal{C}$ is rigid, then every lax $\mathcal{C}$-module functor is strong.
\end{lemma}

\begin{lemma}[{\cite[Lemma 2.11]{2014arXiv1406.4204D}}]
  \label{lem:C-mod-func-adj}
  Let $F: \mathcal{M} \to \mathcal{N}$ be a functor between left $\mathcal{C}$-module categories, and suppose that $F$ has a right adjoint $G: \mathcal{N} \to \mathcal{M}$. Then there is a one-to-one correspondence between colax $\mathcal{C}$-module functor structures on $F$ and lax $\mathcal{C}$-module functor structures on $G$.
\end{lemma}

If $\mathcal{C}$ is rigid, then lax $\mathcal{C}$-module functors and strong $\mathcal{C}$-module functors are simply called $\mathcal{C}$-module functors in view of Lemma~\ref{lem:C-mod-func-rigid}.
Lemma~\ref{lem:C-mod-func-adj} says that, if $\mathcal{C}$ is rigid, then the class of $\mathcal{C}$-module functors is closed under taking adjoints.

\subsection{The centralizer of a monoidal functor}

Let $F: \mathcal{C} \to \mathcal{D}$ be a monoidal functor between monoidal categories $\mathcal{C}$ and $\mathcal{D}$ with structure morphisms
\begin{equation*}
  F_0: \unitobj \to F(\unitobj)
  \quad \text{and} \quad
  F_2(X, Y): F(X) \otimes F(Y) \to F(X \otimes Y)
  \quad (X, Y \in \mathcal{C}).
\end{equation*}
The {\em centralizer} of $F$, originally introduced by Majid \cite{MR1151906} under the name ``the dual category of the functored category'', is a category $\mathcal{Z}(F)$ defined as follows:
An object of $\mathcal{Z}(F)$ is a pair $(V, \sigma)$ consisting of an object $V \in \mathcal{D}$ and a natural isomorphism
\begin{equation*}
  \sigma(X): V \otimes F(X) \to F(X) \otimes V
  \quad (X \in \mathcal{C})
\end{equation*}
such that the equations $\sigma(\unitobj) \circ (\id_V \otimes F_0) = F_0 \otimes \id_V$ and
\begin{equation*}
  \sigma(X \otimes Y) \circ (\id_{V} \otimes F_2(X, Y))
  = (F_2(X, Y) \otimes \id_V) \circ (\id_X \otimes \sigma(Y)) \circ (\sigma(X) \otimes \id_Y)
\end{equation*}
hold for all $X, Y \in \mathcal{C}$. A morphism from $(V, \sigma)$ to $(W, \tau)$ in $\mathcal{Z}(F)$ is a morphism $f: V \to W$ in $\mathcal{D}$ such that
\begin{equation*}
  (\id_{F(X)} \otimes f) \circ \sigma(X) = \tau(X) \circ (f \otimes \id_{F(X)})
\end{equation*}
holds for all $X \in \mathcal{C}$. The centralizer of $\id_{\mathcal{C}}$ is called the (monoidal) {\em center} of $\mathcal{C}$ and denoted by $\mathcal{Z}(\mathcal{C})$. The category $\mathcal{Z}(F)$ is in fact a monoidal category and $\mathcal{Z}(\mathcal{C})$ is moreover a braided monoidal category.

\subsection{Modules over an algebra}

Let $\mathcal{C}$ be a monoidal category. An {\em algebra} in $\mathcal{C}$ is a synonym for a monoid in $\mathcal{C}$ \cite[VII.3]{MR1712872}. Given algebras $A$ and $B$ in $\mathcal{C}$, we denote by ${}_A \mathcal{C}$, $\mathcal{C}_B$ and ${}_A \mathcal{C}_B$ the categories of left $A$-modules, right $B$-modules and $A$-$B$-bimodules in $\mathcal{C}$, respectively. For our purpose, it is convenient to extend the notion of modules over an algebra in the following way (see \cite{MR0498792}).

\begin{definition}
  \label{def:modules}
  Given an algebra $A$ in $\mathcal{C}$, we denote by ${}_A \mathcal{M}$ the Eilenberg-Moore category of the monad $A \ogreaterthan (-)$ on $\mathcal{M}$. An object of the category ${}_A \mathcal{M}$ will be referred to as a {\em left $A$-module in $\mathcal{M}$}. A right $A$-module in a right $\mathcal{C}$-module category and an $A$-$B$-bimodule in a $\mathcal{C}$-bimodule category are also defined analogously.
\end{definition}

Note that $\mathcal{C}$ is a $\mathcal{C}$-bimodule category by the tensor product. The notation and the terminology given in Definition~\ref{def:modules} are consistent with those introduced at the beginning of this subsection.

\subsection{Closed module categories}
\label{subsec:closed-C-mod}

Let $\mathcal{C}$ be a monoidal category, and let $\mathcal{M}$ be a left $\mathcal{C}$-module category. We say that $\mathcal{M}$ is {\em closed} if, for every object $M \in \mathcal{M}$, the functor $\mathcal{C} \to \mathcal{M}$ defined by $X \mapsto X \ogreaterthan M$ has a right adjoint ({\it cf}. the definition of closed monoidal categories). If this is the case, then we denote by $\iHom_{\mathcal{M}}(M, -)$ a right adjoint of the functor $(-) \ogreaterthan M$. By the parameter theorem for adjunctions, the assignment $(M, N) \mapsto \iHom_{\mathcal{M}}(M, N)$ extends to a functor from $\mathcal{M}^{\op} \times \mathcal{M}$ to $\mathcal{C}$ such that there is a natural isomorphism
\begin{equation}
  \label{eq:internal-Hom-adj}
  \Hom_{\mathcal{M}}(X \ogreaterthan M, M') \cong \Hom_{\mathcal{C}}(X, \iHom_{\mathcal{M}}(M, M'))
\end{equation}
for $M, M' \in \mathcal{M}$ and $X \in \mathcal{C}$. The functor $\iHom_{\mathcal{M}}$ is called the {\em internal Hom functor} for $\mathcal{M}$ and makes $\mathcal{M}$ a $\mathcal{C}$-enriched category. For simplicity, we often write $\iHom_{\mathcal{M}}$ as $\iHom$ if $\mathcal{M}$ is obvious from the context.

\subsection{Finite tensor categories}

Given an algebra $A$ over $k$ ($=$ an associative unital algebra over $k$), we denote by $\mbox{\rm mod-}A$ the category of finite-dimensional right $A$-modules.
A $k$-linear category is said to be {\em finite} if it is $k$-linearly equivalent to $\mbox{\rm mod-}A$ for some finite-dimensional algebra $A$ over $k$.

\begin{definition}
  \label{def:FTC}
  A {\em finite tensor category} over $k$ is a rigid monoidal category $\mathcal{C}$ such that $\mathcal{C}$ is a finite abelian category over $k$ and the tensor product of $\mathcal{C}$ is $k$-linear in each variable.
\end{definition}

Unlike \cite{MR2119143} (and like \cite{2013arXiv1312.7188D,2014arXiv1406.4204D}), we do not assume that the unit object of a finite tensor category is a simple object (thus our finite tensor category is in fact a {\em finite multi-tensor category} in the sense of \cite{MR2119143}).

\subsection{Finite module categories}
\label{subsec:fin-mod-cat}

Let $\mathcal{C}$ be a finite tensor category over $k$. In this paper, we mainly deal with the following class of left $\mathcal{C}$-module categories:

\begin{definition}
  \label{def:FTC-mod}
  A {\em finite left $\mathcal{C}$-module category} is a left $\mathcal{C}$-module category $\mathcal{M}$ such that $\mathcal{M}$ is a finite abelian category and the action $\ogreaterthan: \mathcal{C} \times \mathcal{M} \to \mathcal{M}$ is $k$-linear in each variable and right exact in the first variable. Finite right module categories and finite bimodule categories are defined analogously.
\end{definition}

Let $\mathcal{M}$ be a finite left $\mathcal{C}$-module category.
Then the action $\ogreaterthan: \mathcal{C} \times \mathcal{M} \to \mathcal{M}$ is exact in the second variable.
Indeed, $X^* \ogreaterthan (-)$ and ${}^* X \ogreaterthan (-)$ are a left adjoint and a right adjoint of the functor $X \ogreaterthan (-)$, respectively.

It is well-known that a $k$-linear functor between finite abelian categories over $k$ is left (right) exact if and only if it has a left (right) adjoint (a detailed proof is found in \cite[\S1.2]{2014arXiv1406.4204D}). Thus a finite module category is a closed module category.

\begin{example}
  \label{ex:copower}
  Every finite abelian category $\mathcal{M}$ over $k$ is naturally a finite module category over $\mathcal{V} := \mbox{{\rm mod}-}k$ by the action $\bullet$ defined by
  \begin{equation*}
    \Hom_{\mathcal{M}}(V \bullet M, M') \cong \Hom_k(V, \Hom_{\mathcal{M}}(M, M'))
    \quad (V \in \mathcal{V}, M, M' \in \mathcal{M}).
  \end{equation*}
\end{example}

\begin{example}
  If $A$ is an algebra in $\mathcal{C}$, then the category $\mathcal{C}_A$ of right $A$-modules is naturally a finite left $\mathcal{C}$-module category.
\end{example}

Now let $\mathcal{M}$ be a finite left $\mathcal{C}$-module category. We fix $M \in \mathcal{M}$ and consider the functor $Y_M := \iHom_{\mathcal{M}}(M, -)$ from $\mathcal{M}$ to $\mathcal{C}$. The object $M$ is said to be {\em
 $\mathcal{C}$-projective} if $Y_M$ is exact, and is called a {\em $\mathcal{C}$-generator} if $Y_M$ is faithful.
Following \cite{2014arXiv1406.4204D}, $M$ is $\mathcal{C}$-projective if and only if $P \ogreaterthan M \in \mathcal{M}$ is projective for every projective object $P \in \mathcal{C}$.
Thus an {\em exact $\mathcal{C}$-module category} \cite[Definition 3.1]{MR2119143} is precisely a finite $\mathcal{C}$-module category whose every object is $\mathcal{C}$-projective.

The object $A := \iHom(M, M)$ is an algebra in $\mathcal{C}$ and acts on every object of the form $\iHom(M, M')$ from the left. Hence the functor $Y_M$ induces a functor
\begin{equation*}
  K_M: \mathcal{M} \to \mathcal{C}_A, \quad M' \mapsto \iHom(M, M') \quad (M \in \mathcal{C}).
\end{equation*}
The functor $K_M$ is in fact the comparison functor of \eqref{eq:internal-Hom-adj}. Moreover, it is endowed with a structure of a left $\mathcal{C}$-module functor inherited from $Y_M$. Applying the Barr-Beck theorem \cite[VI.7]{MR1712872}, we see that $K_M$ is an equivalence of $\mathcal{C}$-module categories if and only if $M$ is a $\mathcal{C}$-projective $\mathcal{C}$-generator \cite[Theorem 7.10.1]{MR3242743}.
As a consequence, every finite $\mathcal{C}$-module category is of the form $\mathcal{C}_A$ for some algebra $A$ in $\mathcal{C}$ (see also \cite{MR3242743}).

For finite $\mathcal{C}$-module categories $\mathcal{M}$ and $\mathcal{N}$, we denote by $\REX_{\mathcal{C}}(\mathcal{M}, \mathcal{N})$ the category of $k$-linear right exact $\mathcal{C}$-module functors from $\mathcal{M}$ to $\mathcal{N}$. The following variant of the Eilenberg-Watts theorem is important when we deal with finite module categories (see Pareigis \cite{MR0450361,MR0498792,MR0498793} for a proof under a more general setting; see also \cite{2014arXiv1406.4204D}).

\begin{lemma}
  \label{lem:EW-thm}
  For algebras $A$ and $B$ in $\mathcal{C}$, there is an equivalence
  \begin{equation*}
    {}_A \mathcal{C}_B \xrightarrow{\quad \approx \quad} \REX_{\mathcal{C}}(\mathcal{C}_A, \mathcal{C}_B),
    \quad M \mapsto M \otimes_A (-).
  \end{equation*}
\end{lemma}

\section{Modular object}
\label{sec:mod-obj}

\subsection{Ends and coends}

The aim of this section is to introduce a categorical analogue of the modular function on a Hopf algebra, which we call the {\em modular object}. As preliminaries, we recall from \cite{MR1712872} the notion of {\em ends} and {\em coends}. Let $\mathcal{C}$ and $\mathcal{V}$ be categories, and let $S, T: \mathcal{C}^{\op} \times \mathcal{C} \to \mathcal{V}$ be functors. A {\em dinatural transformation} $\xi: S \xrightarrow{\ .. \ } T$ is a family
\begin{equation*}
  \xi = \{ \xi_{X}: S(X,X) \to T(X,X) \}_{X \in \mathcal{C}}
\end{equation*}
of morphisms in $\mathcal{V}$ such that the equation
\begin{equation*}
  T(\id_X, f) \circ \xi_X \circ S(f, \id_X) = T(f, \id_Y) \circ \xi_Y \circ S(\id_Y, f)
\end{equation*}
holds for all morphism $f: X \to Y$ in $\mathcal{C}$. We regard an object $X \in \mathcal{V}$ as the constant functor from $\mathcal{C}^{\op} \times \mathcal{C}$ to $\mathcal{V}$ sending all objects to $X$. Then an {\em end} of $S$ is a pair $(E, p)$ consisting of an object $E \in \mathcal{V}$ and a dinatural transformation $p: E \xrightarrow{\ .. \ } S$ satisfying a certain universal property. Dually, a {\em coend} of $T$ is a pair $(C, i)$ consisting of an object $C \in \mathcal{V}$ and a universal dinatural transformation $i: T \xrightarrow{\ ..\ } C$. The end of $S$ and the coend of $T$ are expressed as
\begin{equation*}
  E = \int_{X \in \mathcal{C}} S(X,X)
  \quad \text{and} \quad
  C = \int^{X \in \mathcal{C}} T(X,X),
\end{equation*}
respectively; see \cite{MR1712872} for more details.

We now suppose that a coend $(C, i)$ of $T: \mathcal{C}^{\op} \times \mathcal{C} \to \mathcal{V}$ exists. If $\mathcal{C}$ has an equivalence $(-)^*: \mathcal{C} \to \mathcal{C}^{\op}$ of categories ({\it e.g.}, when $\mathcal{C}$ is a rigid monoidal category), then the pair $(T, i')$, where $i'_X = i_{X^*}^{}$, is a coend of $(X, Y) \mapsto T(Y^*, X^*)$. This result can be expressed symbolically as follows:
\begin{equation}
  \label{eq:coend-1}
  \int^{X \in \mathcal{C}} T(X,X) = \int^{X \in \mathcal{C}} T(X^*, X^*).
\end{equation}
If $\mathcal{V}$ has an equivalence $(-)^*: \mathcal{V} \to \mathcal{V}^{\op}$, then the pair $(C^*, p)$, where $p_X = (i_X)^*$, is an end of the functor $(X, Y) \mapsto T(Y, X)^*$. Symbolically, we have
\begin{equation}
  \label{eq:coend-2}
  \Big( \int^{X \in \mathcal{C}} T(X,X) \Big)^* = \int_{X \in \mathcal{C}} T(X,X)^*.
\end{equation}

\subsection{The Deligne tensor product of abelian categories}

In what follows, we consider functors between categories whose objects are functors. To avoid confusion, we adopt the following convention:

\begin{notation}
  For a functor $\Psi$ whose source is a category consisting of functors and an object $F$ of the source category, we usually write $\Psi[F]$ instead of $\Psi(F)$.
\end{notation}

For $k$-linear abelian categories $\mathcal{M}$ and $\mathcal{N}$, their {\em Deligne tensor product} \cite[\S5]{MR1106898} is a $k$-linear abelian category $\mathcal{M} \boxtimes \mathcal{N}$ endowed with a functor $\boxtimes: \mathcal{M} \times \mathcal{N} \to \mathcal{M} \boxtimes \mathcal{N}$ that is $k$-linear and right exact in each variable and is universal among such functors out of $\mathcal{M} \times \mathcal{N}$. If $A$ and $B$ are finite-\hspace{0pt}dimensional algebras over $k$, then
\begin{equation}
  \label{eq:Deligne-tensor-fin-ab}
  (\mbox{\rm mod-}A) \boxtimes (\mbox{\rm mod-}B) = \mbox{\rm mod-}(A \otimes_k B)
\end{equation}
with $M \boxtimes N = M \otimes_k N$. Thus, if $\mathcal{M}$ and $\mathcal{N}$ are finite abelian categories, then their Deligne tensor product exists and the functor $\boxtimes: \mathcal{M} \times \mathcal{N} \to \mathcal{M} \boxtimes \mathcal{N}$ is exact in each variable. Moreover, we have a natural isomorphism
\begin{equation}
  \Hom_{\mathcal{M} \boxtimes \mathcal{N}}(M \boxtimes M', N \boxtimes N')
  \cong \Hom_{\mathcal{M}}(M, M') \otimes_k \Hom_{\mathcal{N}}(N, N')
\end{equation}
for $M, M' \in \mathcal{M}$ and $N, N' \in \mathcal{N}$.

For $k$-linear abelian categories $\mathcal{M}$ and $\mathcal{N}$, we denote by $\REX(\mathcal{M}, \mathcal{N})$ the category of $k$-linear right exact functors from $\mathcal{M}$ to $\mathcal{N}$. We also note the following property of the Deligne tensor product.

\begin{lemma}
  \label{lem:Deligne-tensor-rex}
  For finite abelian categories $\mathcal{M}$ and $\mathcal{N}$ over $k$, the following functor $\Phi$ is an equivalence of $k$-linear categories:
  \begin{equation*}
    \Phi: \mathcal{M}^{\op} \boxtimes \mathcal{N} \to \REX(\mathcal{M}, \mathcal{N}),
    \quad M \boxtimes N \mapsto \Hom_{\mathcal{M}}(-, M)^* \bullet N,
  \end{equation*}
  where $\bullet$ is given in Example~\ref{ex:copower}. A quasi-inverse $\overline{\Phi}$ of $\Phi$ is given by
  \begin{equation}
    \label{eq:Deligne-tensor-rex-1}
    \overline{\Phi}: \REX(\mathcal{M}, \mathcal{N}) \to \mathcal{M}^{\op} \boxtimes \mathcal{N},
    \quad F \mapsto \int_{X \in \mathcal{M}} X \boxtimes F(X).
  \end{equation}
\end{lemma}
\begin{proof}
  We may assume that $\mathcal{M} = \mbox{\rm mod-}A$ and $\mathcal{N} = \mbox{\rm mod-}B$ for some finite-\hspace{0pt}dimensional algebras $A$ and $B$ over $k$. Then, by the above-mentioned realization of the Deligne tensor product, the following functor is an equivalence:
  \begin{equation*}
    \mathcal{M}^{\op} \boxtimes \mathcal{N}
    \xrightarrow{\quad \approx \quad} A\mbox{\rm -mod-}B,
    \quad M \boxtimes N \mapsto M^* \otimes_k N,
  \end{equation*}
  where $A\mbox{\rm -mod-}B$ is the category of finite-dimensional $A$-$B$-bimodules. Since both $A$ and $B$ are finite-dimensional, we also have an equivalence
  \begin{equation*}
    A\mbox{\rm -mod-}B \xrightarrow{\quad \approx \quad} \REX(\mathcal{M}, \mathcal{N}),
    \quad X \mapsto (-) \otimes_A X.
  \end{equation*}
  The functor $\Phi$ is an equivalence as the composition of these two equivalences. Now let $\overline{\Phi}$ be a quasi-inverse of $\Phi$. Then we have natural isomorphisms
  \begin{align*}
    \Hom_{\mathcal{M}^{\op} \boxtimes \mathcal{N}}(M \boxtimes N, \overline{\Phi}[F])
    & \cong \NAT(\Phi(M \boxtimes N), F) \\
    & \cong \textstyle \int_{X \in \mathcal{M}} \Hom_{\mathcal{N}}(\Hom_{\mathcal{M}}(X, M)^* \bullet N, F(X)) \\
    & \cong \textstyle \int_{X \in \mathcal{M}} \Hom_{k}(\Hom_{\mathcal{M}}(X, M)^*, \Hom_{\mathcal{N}}(N, F(X))) \\
    & \cong \textstyle \int_{X \in \mathcal{M}} \Hom_{\mathcal{M}}(X, M) \otimes \Hom_{\mathcal{N}}(N, F(X)) \\
    & \cong \textstyle \int_{X \in \mathcal{M}} \Hom_{\mathcal{M}^{\op} \boxtimes \mathcal{N}}(M \boxtimes N, X \boxtimes F(X))
  \end{align*}
  for $M \in \mathcal{M}$, $N \in \mathcal{N}$ and $F \in \REX(\mathcal{M}, \mathcal{N})$. Since every object of $\mathcal{M}^{\op} \boxtimes \mathcal{N}$ is a finite colimit of objects of the form $M \boxtimes N$, the above computation implies that $\overline{\Phi}[F]$ represents the functor
  \begin{equation*}
    (\mathcal{M}^{\op} \boxtimes \mathcal{N})^{\op} \to \mathbf{Vec},
    \quad L \mapsto \int_{X \in \mathcal{M}} \Hom_{\mathcal{M}^{\op} \boxtimes \mathcal{N}}(L, X \boxtimes F(X)),
  \end{equation*}
  where $\mathbf{Vec}$ is the category of all vector spaces over $k$. Hence the end in~\eqref{eq:Deligne-tensor-rex-1} indeed exists and is isomorphic to $\overline{\Phi}[F]$.
\end{proof}

\subsection{Tensor product of module categories}

Let $\mathcal{C}$ and $\mathcal{D}$ be finite tensor categories over a field $k$. Then $\mathcal{E} := \mathcal{C} \boxtimes \mathcal{D}^{\rev}$ is a monoidal category with the component-wise tensor product. The monoidal category $\mathcal{E}$ is not rigid in general. We note that $\mathcal{E}$ is rigid (and therefore a finite tensor category) if, for example, the base field $k$ is perfect \cite[\S5]{MR1106898}.

By a finite $\mathcal{C}$-$\mathcal{D}$-bimodule category, we mean a $\mathcal{C}$-$\mathcal{D}$-bimodule category that is finite both as a left $\mathcal{C}$- and a right $\mathcal{D}$-module category. Although $\mathcal{E}$ is not a finite tensor category in general, we abuse terminology and say that an $\mathcal{E}$-module category $\mathcal{M}$ is finite if it is a finite abelian category and the action of $\mathcal{E}$ on $\mathcal{M}$ is $k$-linear and right exact in each variable. By the universal property of the Deligne tensor product, a finite $\mathcal{C}$-$\mathcal{D}$-bimodule category can be identified with a finite $\mathcal{E}$-module category.

Now let $\mathcal{M}$ and $\mathcal{N}$ be finite module categories over $\mathcal{D}$ and $\mathcal{C}$, respectively. Then $\mathcal{M}^{\op} \boxtimes \mathcal{N}$ is a finite $\mathcal{C}$-$\mathcal{D}$-bimodule category by the action determined by
\begin{equation}
  X \ogreaterthan (M \boxtimes N) \olessthan Y
  = ({}^*Y \ogreaterthan M) \boxtimes (X \ogreaterthan N)
\end{equation}
for $X \in \mathcal{C}$, $Y \in \mathcal{D}$, $M \in \mathcal{M}$ and $N \in \mathcal{N}$. $\REX(\mathcal{M}, \mathcal{N})$ is also a finite $\mathcal{C}$-$\mathcal{D}$-bimodule category by the action determined by
\begin{equation}
  \label{eq:C-env-act-rex}
  (X \ogreaterthan F \olessthan Y)(M)
  = X \ogreaterthan F(Y \ogreaterthan M)
\end{equation}
for $X \in \mathcal{C}$, $Y \in \mathcal{D}$, $M \in \mathcal{M}$ and $F \in \REX(\mathcal{M}, \mathcal{N})$. The proof of the following lemma is straightforward.

\begin{lemma}
  The equivalence $\Phi: \mathcal{M}^{\op} \boxtimes \mathcal{N} \to \REX(\mathcal{M}, \mathcal{N})$ given in Lemma~\ref{lem:Deligne-tensor-rex} is an equivalence of $\mathcal{C}$-$\mathcal{D}$-bimodule categories.
\end{lemma}

The following technical lemma will be used in Section~\ref{sec:braided-Hopf}.

\begin{lemma}
  \label{lem:Rex-F-G}
  Let $G: \mathcal{M}_2 \to \mathcal{M}_1$ and $E: \mathcal{N}_1 \to \mathcal{N}_2$ be $k$-linear right exact module functors between finite module categories over $\mathcal{D}$ and $\mathcal{C}$, respectively. Then
  \begin{equation*}
    \REX(G, E): \REX(\mathcal{M}_1, \mathcal{N}_1) \to \REX(\mathcal{M}_2, \mathcal{N}_2),
    \quad F \mapsto E \circ F \circ G
  \end{equation*}
  is a $k$-linear right exact strong $\mathcal{E}$-module functor.
\end{lemma}
\begin{proof}
  Let $R: \mathcal{M}_1 \to \mathcal{M}_2$ be a right adjoint functor of $G$. Since $R$ is $k$-linear and {\em left} exact as a right adjoint, the functor $R^{\op}: \mathcal{M}_1^{\op} \to \mathcal{M}_2^{\op}$ induced by $R$ is $k$-linear and {\em right} exact. Hence we have a $k$-linear right exact functor
  \begin{equation*}
    H := R^{\op} \boxtimes E: \mathcal{M}_1^{\op} \boxtimes \mathcal{N}_1 \to \mathcal{M}_2^{\op} \boxtimes \mathcal{N}_2.
  \end{equation*}
  The functor $H$ is obviously a strong $\mathcal{E}$-module functor. Now we compute
  \begin{align*}
    \Big( \REX(G, E) \Big[ \Phi(M_1 \boxtimes N_1) \Big] \Big)(M_2)
    & \cong \Hom_{\mathcal{M}_1}(G(M_2), M_1)^* \bullet E(N_1) \\
    & \cong \Hom_{\mathcal{M}_2}(M_2, R(M_1))^* \bullet E(N_1) \\
    & = \Phi(H(M_1 \boxtimes N_1))(M_2)
  \end{align*}
  for $M_1 \in \mathcal{M}_1^{\op}$, $M_2 \in \mathcal{M}_2$ and $N_1 \in \mathcal{N}_1$, where $\Phi$'s are the equivalences given in Lemma~\ref{lem:Deligne-tensor-rex}. Thus $\REX(G, E) \cong \Phi \circ H \circ \overline{\Phi}$. The functor $\REX(G, E)$ has the desired properties since $H$ does.
\end{proof}

\subsection{Monadic description of module functors}
\label{subsec:mona-mod-fun}

Let $\mathcal{C}$ be a finite tensor category over $k$, and set $\mathcal{C}^{\env} = \mathcal{C} \boxtimes \mathcal{C}^{\rev}$. We define $A \in \mathcal{C}^{\env}$ by
\begin{equation*}
  A = \int^{X \in \mathcal{C}} X^* \boxtimes X
\end{equation*}
(see \cite[\S5]{MR1862634} or \cite{2014arXiv1402.3482S} for the existence of this coend). Let $i_X: X^* \boxtimes X \to A$ be the universal dinatural transformation for the coend. By the Fubini theorem for coends, $A \otimes A$ is a coend of $(X_1, X_2, Y_1, Y_2) \mapsto (X_1^* \boxtimes Y_1) \otimes (X_2^* \boxtimes Y_2)$. Thus there exists a unique morphism $m$ such that the diagram
\begin{equation*}
  \xymatrix{
    (X^* \boxtimes X) \otimes (Y^* \boxtimes Y)
    \ar[rrr]^(.6){i_X \otimes i_Y} \ar@{=}[d]
    & & & A \otimes A \ar[d]^{m} \\
    (Y \otimes X)^* \boxtimes (Y \otimes X)
    \ar[rrr]^(.6){i_{Y \otimes X}}
    & & & A
  }
\end{equation*}
commutes for all $X, Y \in \mathcal{C}$. We also define $u: \unitobj \boxtimes \unitobj \to A$ by $u = i_{\unitobj}$. The proof of the following lemma is straightforward:

\begin{lemma}
  \label{lem:coend-algebra}
  The triple $(A, m, u)$ is an algebra in $\mathcal{C}^{\env}$.
\end{lemma}

Let $\mathcal{M}$ and $\mathcal{N}$ be finite left $\mathcal{C}$-module categories. As we have seen, $\REX(\mathcal{M}, \mathcal{N})$ is a left $\mathcal{C}^{\env}$-module category by the action given by \eqref{eq:C-env-act-rex}. We now consider the category of left $A$-modules in $\REX(\mathcal{M}, \mathcal{N})$ in the sense of Definition~\ref{def:modules}.

\begin{lemma}
  \label{lem:Rex-C-as-modules}
  ${}_A \REX(\mathcal{M}, \mathcal{N}) \cong \REX_{\mathcal{C}}(\mathcal{M}, \mathcal{N})$.
\end{lemma}
\begin{proof}
  Day and Street \cite{MR2342829} showed that the functor $Z(V) = \int^{X \in \mathcal{C}} X^* \otimes V \otimes X$ ($V \in \mathcal{C}$) has a structure of a monad such that the category of $Z$-modules can be identified with the monoidal center of $\mathcal{C}$. The proof is essentially same as their proof of this fact: For $F \in \REX(\mathcal{M}, \mathcal{N})$, we have
  \begin{align*}
    \NAT(A \ogreaterthan F, F)
    & \textstyle \cong \int_{M \in \mathcal{M}, X \in \mathcal{C}} \Hom_{\mathcal{N}}(X^* \ogreaterthan F(X \ogreaterthan M), F(M)) \\
    & \textstyle \cong \int_{M \in \mathcal{M}, X \in \mathcal{C}} \Hom_{\mathcal{N}}(F(X \ogreaterthan M), X \ogreaterthan F(M))
  \end{align*}
  by the standard properties of ends and coends. Hence a morphism $\mu: A \ogreaterthan F \to F$ in $\REX(\mathcal{M}, \mathcal{N})$ is the same thing as a natural transformation
  \begin{equation*}
    \xi_{M,X}: F(X \ogreaterthan M) \to X \ogreaterthan F(M)
    \quad (M \in \mathcal{M}, X \in \mathcal{C}).
  \end{equation*}
  The morphism $\mu$ makes $F$ an $A$-module in $\REX(\mathcal{M}, \mathcal{N})$ if and only if the corresponding natural transformation $\xi$ makes $F$ a colax $\mathcal{C}$-module functor. Hence, by Lemma~\ref{lem:C-mod-func-rigid}, we obtain a bijection between the objects of the two categories, which gives rise to an isomorphism of the two categories.
\end{proof}

We write $\REX(\mathcal{C}, \mathcal{C})$ as $\REX(\mathcal{C})$ for short. Lemma~\ref{lem:Deligne-tensor-rex} yields an equivalence
\begin{equation}
  \label{eq:equiv-Phi-1}
  \Phi_{\mathcal{C}}: \mathcal{C}^{\env} \to \REX(\mathcal{C}),
  \quad V \boxtimes W \mapsto \Hom_{\mathcal{C}}(-, {}^*W)^* \bullet V
  \quad (V, W \in \mathcal{C})
\end{equation}
of left $\mathcal{C}^{\env}$-module categories with quasi-inverse
\begin{equation}
  \label{eq:equiv-Phi-inv}
  \overline{\Phi}_{\mathcal{C}}: \REX(\mathcal{C}) \to \mathcal{C}^{\env},
  \quad F \mapsto \int_{X \in \mathcal{C}} F(X) \boxtimes X^*.
\end{equation}

\begin{lemma}
  \label{lem:Hopf-bimod-FTC-rex}
  The following functor is an equivalence of categories:
  \begin{equation}
    \label{eq:Hopf-bimod-FTC-rex}
    \mathfrak{H}_{\mathcal{C}}: {}_A(\mathcal{C}^{\env}) \to \mathcal{C},
    \quad M \mapsto \Phi_{\mathcal{C}}(M)(\unitobj).
  \end{equation}
\end{lemma}
\begin{proof}
  $\Phi_{\mathcal{C}}$ induces an equivalence between the categories of $A$-modules in $\mathcal{C}^{\env}$ and those in $\REX(\mathcal{C})$. The functor $\mathfrak{H}_{\mathcal{C}}$ is an equivalence, since it is the composition
  \begin{equation*}
    {}_A (\mathcal{C}^{\env})
    \xrightarrow[\quad \Phi_{\mathcal{C}} \quad]{\approx}
    {}_A \REX(\mathcal{C})
    \xrightarrow[\quad \text{Lemma~\ref{lem:Rex-C-as-modules}} \quad]{\cong}
    \REX_{\mathcal{C}}(\mathcal{C})
    \xrightarrow[\qquad]{\approx}
    \mathcal{C},
  \end{equation*}
  where the last arrow refers to the functor $F \mapsto F(\unitobj)$.
\end{proof}

\begin{remark}
  \label{rem:Hopf-bimod-FTC-lex}
  A similar equivalence can be obtained by using {\em left} exact functors instead of right exact ones.
  The Deligne tensor product of finite abelian categories $\mathcal{A}$ and $\mathcal{B}$ over $k$ also has the universal property for functors out of $\mathcal{A} \times \mathcal{B}$ that is $k$-linear and {\em left} exact in each variable \cite[\S5]{MR1106898}. Using this universal property, we define a $k$-linear {\em left} exact functor
  \begin{equation}
    \label{eq:equiv-Phi'-1}
    \Phi'_{\mathcal{C}}: \mathcal{C}^{\env} \to \LEX(\mathcal{C}),
    \quad V \boxtimes W \mapsto \Hom_{\mathcal{C}}(W^*, -) \bullet V,
  \end{equation}
  where $\LEX(\mathcal{C})$ is the category of $k$-linear {\em left} exact endofunctors on $\mathcal{C}$. As explained in \cite{2014arXiv1402.3482S}, the functor $\Phi'_{\mathcal{C}}$ is in fact an equivalence with quasi-inverse
  \begin{equation}
    \label{eq:equiv-Phi'-2}
    \overline{\Phi}{}'_{\mathcal{C}}: \LEX(\mathcal{C}) \to \mathcal{C}^{\env},
    \quad F \mapsto \int^{X \in \mathcal{C}} F(X^*) \boxtimes X.
  \end{equation}
  The monoidal category $\mathcal{C}^{\env}$ acts on $\LEX(\mathcal{C})$ from the left in such a way that $\Phi_{\mathcal{C}}'$ is an equivalence of left $\mathcal{C}^{\env}$-module categories. By the same argument as Lemma~\ref{lem:Rex-C-as-modules}, we can identify ${}_A \LEX(\mathcal{C})$ with the category $\LEX_{\mathcal{C}}(\mathcal{C})$ of $k$-linear left exact left $\mathcal{C}$-module functors on $\mathcal{C}$. The functor
  \begin{equation}
    \mathfrak{H}'_{\mathcal{C}}: {}_A(\mathcal{C}^{\env}) \to \mathcal{C},
    \quad M \mapsto \Phi'_{\mathcal{C}}(M)(\unitobj)
  \end{equation}
  is an equivalence of categories since it is obtained as the composition
  \begin{equation}
    \label{eq:pf-Hopf-bimod-FTC-lex-1}
    {}_{A}(\mathcal{C}^{\env})
    \xrightarrow{\quad \Phi'_{\mathcal{C}} \quad}
    {}_A \LEX(\mathcal{C})
    \xrightarrow{\quad \cong \quad}
    \LEX_{\mathcal{C}}(\mathcal{C})
    \xrightarrow{\quad \approx \quad}
    \mathcal{C}.
  \end{equation}
  By~\eqref{eq:equiv-Phi'-2} and \eqref{eq:pf-Hopf-bimod-FTC-lex-1}, a quasi-inverse of $\mathfrak{H}'_{\mathcal{C}}$ is given by
  \begin{equation}
    \overline{\mathfrak{H}}{}'_{\mathcal{C}}:
    \mathcal{C} \to {}_A (\mathcal{C}^{\env}),
    \quad V \mapsto A \otimes (V \boxtimes \unitobj).
  \end{equation}
\end{remark}

\begin{remark}
  \label{rem:comparison-ENO}
  Suppose that $\mathcal{C}^{\env}$ is rigid. Then $\mathcal{C}^{\env}$ is a finite tensor category acting on $\mathcal{C}$ from the left by $(X \boxtimes Y) \ogreaterthan V = X \otimes V \otimes Y$. Let $\iHom$ denote the associated internal Hom functor. As shown in \cite{2014arXiv1402.3482S}, the algebra $\iHom(\unitobj, \unitobj)$ is canonically isomorphic to the algebra $A$ of Lemma~\ref{lem:coend-algebra}. Hence, by the result recalled in \S\ref{subsec:fin-mod-cat}, we obtain an equivalence
  \begin{equation}
    \label{eq:Hopf-bimod-ENO}
    \mathfrak{K}: \mathcal{C} \to (\mathcal{C}^{\env})_A,
    \quad V \mapsto \iHom(\unitobj, V) \cong (V \boxtimes \unitobj) \otimes A
  \end{equation}
  of left $\mathcal{C}^{\env}$-module categories \cite[Proposition 2.3]{MR2097289}. The equivalences $\mathfrak{K}$, as well as $\mathfrak{H}_{\mathcal{C}}$ and $\mathfrak{H}_{\mathcal{C}}'$, can be thought of as category-theoretical variants of the fundamental theorem for Hopf bimodules. There is the following relation:
  \begin{equation}
    \label{eq:comparison-ENO}
    \overline{\mathfrak{H}}_{\mathcal{C}}(V) \cong {}^* \! \mathfrak{K}(V^*)
    \quad (V \in \mathcal{C}),
  \end{equation}
  where $\overline{\mathfrak{H}}_{\mathcal{C}}$ is a quasi-inverse of $\mathfrak{H}_{\mathcal{C}}$.
  Indeed, by \eqref{eq:coend-2} and~\eqref{eq:equiv-Phi-inv}, we have
  \begin{equation*}
    {}^*A \otimes (V \boxtimes \unitobj)
    \cong \int_X {}^*(X^* \boxtimes X) \otimes (V \boxtimes \unitobj)
    \cong \int_X (X \otimes V) \boxtimes X^*
    \cong \overline{\Phi}_{\mathcal{C}} \Big[ (-) \otimes V \Big].
  \end{equation*}
  Hence $\mathfrak{H}_{\mathcal{C}}({}^* \! \mathfrak{K}(V^*)) \cong \mathfrak{H}_{\mathcal{C}}({}^* \! A \otimes (V \boxtimes \unitobj)) \cong \unitobj \otimes V = V$ for all $V \in \mathcal{C}$.
\end{remark}

\subsection{The modular object}
\label{subsec:modular-obj}

Let $\mathcal{C}$ be a finite tensor category over $k$. We consider the (right) {\em Cayley functor} defined by
\begin{equation}
  \label{eq:mod-obj-Cayley}
  \Upsilon_{\mathcal{C}}: \mathcal{C} \to \REX(\mathcal{C}),
  \quad X \mapsto (-) \otimes X
  \quad (X \in \mathcal{C}).
\end{equation}
If we identify $\mathcal{C}$ with $\REX_{\mathcal{C}}(\mathcal{C})$, then $\Upsilon_{\mathcal{C}}$ corresponds to the forgetful functor from $\REX_{\mathcal{C}}(\mathcal{C})$ and thus it has a left adjoint, say $\Upsilon_{\mathcal{C}}^*$ (see also \eqref{eq:Cayley-and-Phi} below).

\begin{definition}
  \label{def:mod-obj}
  The {\em modular object} $\alpha_{\mathcal{C}}^{} \in \mathcal{C}$ is defined to be the image of
  \begin{equation}
    \label{eq:mod-obj-triv-bimod}
    J_{\mathcal{C}} := \Hom_{\mathcal{C}}(-, \unitobj)^* \bullet \unitobj \in \REX(\mathcal{C})
  \end{equation}
  under a left adjoint of the Cayley functor. Namely,
  \begin{equation}
    \label{eq:mod-obj-def}
    \alpha_{\mathcal{C}}^{} := \Upsilon^*_{\mathcal{C}}[J^{}_{\mathcal{C}}].
  \end{equation}
\end{definition}

If $\mathcal{C}$ is the representation category of a finite-dimensional Hopf algebra $H$, then the functor $J_{\mathcal{C}}$ is given by tensoring the trivial $H$-bimodule. The functor $\Upsilon_{\mathcal{C}}^*$ can be described by the fundamental theorem for Hopf bimodules. Consequently, the modular object corresponds to the modular function on $H$. The detail will be discussed in Section~\ref{sec:braided-Hopf} in a more general setting.

To study the properties of the modular object, we clarify relations between the Cayley functor and other functors appeared in the previous subsection. Let $\Phi_{\mathcal{C}}$ be the equivalence given by~\eqref{eq:equiv-Phi-1}. Then the diagram
\begin{equation}
  \label{eq:Cayley-and-Phi}
  \xymatrix{
    {}_A (\mathcal{C}^{\env})
    \ar[d]
    \ar[rr]^{\Phi_{\mathcal{C}}}
    & & {}_A \REX(\mathcal{C})
    \ar[d]
    \ar[rr]^{\cong}_{\text{Lemma~\ref{lem:Rex-C-as-modules}}}
    & & \REX_{\mathcal{C}}(\mathcal{C})
    \ar[d]
    \ar[r]^(.6){\approx}
    & \mathcal{C}
    \ar[d]^{\Upsilon_{\mathcal{C}}} \\
    \mathcal{C}^{\env}
    \ar[rr]^{\Phi_{\mathcal{C}}}
    & & \REX(\mathcal{C})
    \ar@{=}[rr]
    & & \REX(\mathcal{C}) \ar@{=}[r]
    & \REX(\mathcal{C})
  }
\end{equation}
commutes up to isomorphisms, where $A$ is the algebra introduced in Lemma~\ref{lem:coend-algebra} and the unlabeled vertical arrows are the forgetful functors. By considering left adjoints of functors in this diagram, we have
\begin{equation}
  \label{eq:Cayley-left-adj-F}
  \Upsilon_{\mathcal{C}}^*[F]
  = (A \ogreaterthan F)(\unitobj)
  = \int^{X \in \mathcal{C}} X^* \otimes F(X)
  \overset{\text{\eqref{eq:coend-1}}}{=} \int^{X \in \mathcal{C}} X \otimes F({}^*X)
\end{equation}
for $F \in \REX(\mathcal{C})$. Since the composition along the top row of the diagram is the equivalence $\mathfrak{H}_{\mathcal{C}}$ given in Lemma~\ref{lem:Hopf-bimod-FTC-rex}, we also have
\begin{equation}
  \label{eq:Cayley-left-adj-M}
  \Upsilon_{\mathcal{C}}^* \big[ \Phi_{\mathcal{C}}(M) \big] \cong \mathfrak{H}_{\mathcal{C}}(A \otimes M)
  \quad (M \in \mathcal{C}^{\env}).
\end{equation}
Hence, in particular,
\begin{equation}
  \label{eq:Cayley-left-adj-2}
  \alpha_{\mathcal{C}}
  = \Upsilon_{\mathcal{C}}^* \big[ J_{\mathcal{C}} \big]
  = \Upsilon_{\mathcal{C}}^* \big[ \Phi_{\mathcal{C}}(\unitobj \boxtimes \unitobj) \big] 
  \cong \mathfrak{H}_{\mathcal{C}}(A).
\end{equation}

For a while, we suppose that $\mathcal{C}^{\env}$ is rigid. Let $\mathfrak{K}: \mathcal{C} \to (\mathcal{C}^{\env})_A$ be the equivalence given by~\eqref{eq:Hopf-bimod-ENO}. Then the distinguished invertible object \cite{MR2097289} is defined as the unique (up to isomorphism) object $D \in \mathcal{C}$ such that $\mathfrak{K}(D) \cong A^*$.

\begin{proposition}
  \label{prop:mod-obj-D}
  Under the above assumption, the modular object of $\mathcal{C}$ is isomorphic to the dual of the distinguished invertible object of $\mathcal{C}$.
\end{proposition}
\begin{proof}
  By \eqref{eq:comparison-ENO} and~\eqref{eq:Cayley-left-adj-2}, we compute ${}^* \! D \cong \mathfrak{H}_{\mathcal{C}}({}^* \! \mathfrak{K}(D)) \cong \mathfrak{H}_{\mathcal{C}}(A) \cong \alpha_{\mathcal{C}}$.
\end{proof}

As its name suggests, the distinguished invertible object $D \in \mathcal{C}$ is an invertible object; see \cite{MR2097289} and \cite{2014arXiv1402.3482S}. We remark that the rigidity of $\mathcal{C}^{\env}$ is essentially used in these papers. On the other hand, Definition~\ref{def:mod-obj} makes sense even in the case where $\mathcal{C}^{\env}$ is not rigid.

We now go back to the general situation and prove the invertibility of $\alpha_{\mathcal{C}}$ without assuming the rigidity of $\mathcal{C}^{\env}$.
Our proof uses not only $\mathfrak{H}_{\mathcal{C}}$ but also the equivalence $\mathfrak{H}'_{\mathcal{C}}$ given in Remark~\ref{rem:Hopf-bimod-FTC-lex}.
As before, we denote its quasi-inverse by $\overline{\mathfrak{H}}{}'_{\mathcal{C}}$.

\begin{proposition}
  \label{prop:mod-obj-invert}
  The modular object $\alpha_{\mathcal{C}}$ is an invertible object.
\end{proposition}
\begin{proof}
  By \eqref{eq:Cayley-left-adj-F} and \eqref{eq:Cayley-left-adj-M}, we compute
  \begin{align*}
    \mathfrak{H}_{\mathcal{C}}^{} \overline{\mathfrak{H}}{}'_{\mathcal{C}}(V)
    = \mathfrak{H}_{\mathcal{C}}^{}(A \otimes (V \boxtimes \unitobj))
    \cong \Upsilon_{\mathcal{C}}^* \big[\Phi_{\mathcal{C}}^{}(V \boxtimes \unitobj) \big]
    \cong \alpha_{\mathcal{C}} \otimes V
  \end{align*}
  for all $V \in \mathcal{C}$. Thus the functor $V \mapsto \alpha_{\mathcal{C}} \otimes V$ is an equivalence as the composition of equivalences. This means that $\alpha_{\mathcal{C}}$ is an invertible object.
\end{proof}

We also note the following formula for the modular object:

\begin{proposition}
  \label{prop:mod-obj-coend-formula}
  $\displaystyle \alpha_{\mathcal{C}}^{} = \int^{X \in \mathcal{C}} \Hom_{\mathcal{C}}(\unitobj, X)^* \bullet X$.
\end{proposition}

This formula directly follows from \eqref{eq:Cayley-left-adj-F} and the definition of $\alpha_{\mathcal{C}}$.
Since it does not involve any information about the tensor product of $\mathcal{C}$, we have:

\begin{corollary}
  \label{cor:mod-obj-invariance}
  Let $F: \mathcal{C} \to \mathcal{D}$ be a $k$-linear functor between finite tensor categories $\mathcal{C}$ and $\mathcal{D}$ over $k$ that is an equivalence between underlying categories and satisfies $F(\unitobj) \cong \unitobj$. Then we have $F(\alpha_{\mathcal{C}}) \cong \alpha_{\mathcal{D}}$.
\end{corollary}

This corollary may be useful to find the modular object. For example:

\begin{corollary}
  For a finite tensor category $\mathcal{C}$ over $k$, we have
  \begin{equation*}
    \alpha_{\mathcal{C}^{\rev}}^{} \cong \alpha_{\mathcal{C}}^{}
    \quad \text{and} \quad
    \alpha_{\mathcal{C}^{\op}}^{} \cong \alpha_{\mathcal{C}}^{*}.
  \end{equation*}
\end{corollary}
\begin{proof}
  Apply Corollary~\ref{cor:mod-obj-invariance} to $\id_{\mathcal{C}}: \mathcal{C} \to \mathcal{C}^{\rev}$ and $(-)^*: \mathcal{C} \to \mathcal{C}^{\op}$.
\end{proof}

\subsection{The Radford $S^4$-formula}

Let $\mathcal{C}$ be a finite tensor category over $k$. Here we prove the following analogue of Radford's formula of the fourth power of the antipode of a finite-dimensional Hopf algebra:

\begin{theorem}
  \label{thm:Radford-S4}
  There is an isomorphism of monoidal functors
  \begin{equation}
    \label{eq:Radford-S4}
    {}^{**} X \cong \alpha_{\mathcal{C}} \otimes X^{**}
    \otimes \alpha_{\mathcal{C}}^* \quad (X \in \mathcal{C}).
  \end{equation}
\end{theorem}

We call \eqref{eq:Radford-S4} the {\em Radford $S^4$-formula}. This theorem is first established by Etingof, Nikshych and Ostrik \cite{MR2097289} under the assumption that the base field $k$ is algebraically closed and $\End_{\mathcal{C}}(\unitobj) \cong k$. Other existing proofs \cite{2013arXiv1312.7188D,2014arXiv1402.3482S} work under milder assumptions but seem to heavily rely on the rigidity of $\mathcal{C}^{\env}$. We give a proof of \eqref{eq:Radford-S4} that works without any assumptions.

A key observation for proving \eqref{eq:Radford-S4} is that $\mathcal{C}^{\env}$ acts on $\REX(\mathcal{C})$ not only from the left but also from the right by the action determined by
\begin{equation}
  \label{eq:C-env-r-act-Rex-C}
  F \olessthan (X \boxtimes Y) = F(- \otimes {}^{**}Y) \otimes X
  \quad (F \in \REX(\mathcal{C}), X, Y \in \mathcal{C}).
\end{equation}
One can check that the equivalence $\Phi_{\mathcal{C}}: \mathcal{C}^{\env} \to \REX(\mathcal{C})$ given by~\eqref{eq:equiv-Phi-1} is in fact an equivalence of $\mathcal{C}^{\env}$-bimodule categories.
Hence $\mathcal{C}^{\env}$ also acts on ${}_A \REX(\mathcal{C})$ from the right by~\eqref{eq:C-env-r-act-Rex-C}.

We consider the category $\REX(\mathcal{C})_A$ of right $A$-modules in $\REX(\mathcal{C})$. To describe this category, we introduce the following notation: Given a right $\mathcal{C}$-module category $\mathcal{M}$ and a strong monoidal functor $T: \mathcal{C} \to \mathcal{C}$, we denote by $\mathcal{M}_{\langle T \rangle}$ the category $\mathcal{M}$ with the action twisted by $T$.

\begin{lemma}
  \label{lem:Rex-C-A}
  $\REX(\mathcal{C})_A$ is isomorphic to the category of $k$-linear right $\mathcal{C}$-module functors from $\mathcal{C}_{\langle S^{-2} \rangle}$ to $\mathcal{C}_{\langle S^{2} \rangle}$, where $S = (-)^*$ is the left duality functor on $\mathcal{C}$.
\end{lemma}
\begin{proof}
  The proof of this lemma is almost the same as Lemma~\ref{lem:Rex-C-as-modules} (and thus it is essentially same as the argument due to Day and Street). We first note that a $k$-linear right $\mathcal{C}$-module functor $\mathcal{C}_{\langle S^{-2} \rangle} \to \mathcal{C}_{\langle S^{2} \rangle}$ is automatically exact. Indeed, if $F$ is such a functor, then
  \begin{equation*}
    F(X) = F(\unitobj \otimes {}^{**}X^{**})
    = F(\unitobj \olessthan X^{**}) \cong F(\unitobj) \olessthan X^{**} = F(\unitobj) \otimes X^{****}
  \end{equation*}
  for all $X \in \mathcal{C}$. Now let $F \in \REX(\mathcal{C})$. Then we have
  \begin{align*}
    \NAT(F \olessthan A, F)
    & \textstyle \cong \int_{V, X \in \mathcal{C}} \Hom_{\mathcal{C}}(F(V \otimes {}^{**}X) \otimes X^*, F(V)) \\
    & \textstyle \cong \int_{V, X \in \mathcal{C}} \Hom_{\mathcal{C}}(F(V \otimes {}^{**}X), F(V) \otimes X^{**}).
  \end{align*}
  Thus a morphism $\mu: F \olessthan A \to F$ is the same thing as a natural transformation
  \begin{equation*}
    \xi_{V,X}: F(V \otimes {}^{**}X) \to F(V) \otimes X^{**}
    \quad (X, V \in \mathcal{C}).
  \end{equation*}
  The morphism $\mu$ makes $F$ a right $A$-module in $\REX(\mathcal{C})$ if and only if $\xi$ makes $F$ a colax $\mathcal{C}$-module functor from $\mathcal{C}_{\langle S^{-2} \rangle}$ to $\mathcal{C}_{\langle S^2 \rangle}$. This correspondence gives rise to an isomorphism of the two categories.
\end{proof}

\begin{proof}[Proof of Theorem~\ref{thm:Radford-S4}]
  Set $G = (-) \otimes \alpha_{\mathcal{C}}$. By~\eqref{eq:Cayley-left-adj-2} and the definition of the equivalence $\mathfrak{H}_{\mathcal{C}}$, we have $G \cong (-) \otimes \mathfrak{H}_{\mathcal{C}}(A) \cong \Phi_{\mathcal{C}}(A)$. Since $A$ is a right $A$-module in $\mathcal{C}^{\env}$, $G$ is in fact a right $A$-module in $\REX(\mathcal{C})$. Thus, by the previous lemma, there is a natural isomorphism
  \begin{equation*}
    \xi_{V,X}: G(V \otimes {}^{**}X) \to G(V) \otimes X^{**} \quad (V, X \in \mathcal{C})
  \end{equation*}
  such that $(G, \xi)$ is a right $\mathcal{C}$-module functor from $\mathcal{C}_{\langle S^{-2} \rangle}$ to $\mathcal{C}_{\langle S^{2} \rangle}$. Since $\alpha_{\mathcal{C}}$ is an invertible object, the composition
  \begin{equation}
    \label{eq:Radford-S4-pf-1}
    {}^{**} X
    \xrightarrow{\quad \id \otimes \coev \quad}
    {}^{**} X \otimes \alpha_{\mathcal{C}} \otimes \alpha_{\mathcal{C}}^*
    \xrightarrow{\quad \xi_{\unitobj, X}  \otimes \id \quad}
    \alpha_{\mathcal{C}} \otimes X^{**} \otimes \alpha_{\mathcal{C}}^*
  \end{equation}
  is an isomorphism for all $X \in \mathcal{C}$. The definition of module functors implies that \eqref{eq:Radford-S4-pf-1} is indeed a morphism of monoidal functors.
\end{proof}

\section{The relative modular object}
\label{sec:rel-mod-obj}

\subsection{Tensor functors}
\label{subsec:ten-fun}

By a {\em tensor category}, we mean a $k$-linear abelian category endowed with a rigid monoidal structure such that the tensor product is bilinear (thus a finite tensor category is precisely a tensor category whose underlying category is a finite abelian category). For a $k$-linear functor $T: \mathcal{C} \to \mathcal{D}$ between tensor categories $\mathcal{C}$ and $\mathcal{D}$, we define
\begin{equation*}
  T^!(X) = {}^*T(X^*)
  \quad \text{and} \quad
  {}^!T(X) = T({}^*X)^*
  \quad (X \in \mathcal{C}).
\end{equation*}
Now let $F: \mathcal{C} \to \mathcal{D}$ and $G: \mathcal{D} \to \mathcal{C}$ be $k$-linear functors. The following easy lemma will be used frequently: 

\begin{lemma}
  \label{lem:ten-cat-adj-1}
  $F \dashv G$ implies $G^! \dashv F^!$ and ${}^!G \dashv {}^!F$.
\end{lemma}

By a {\em tensor functor}, we mean a $k$-linear strong monoidal functor between tensor categories. We note that a tensor functor $F: \mathcal{C} \to \mathcal{D}$ preserves the duality. Thus, we have $F^! \cong F \cong {}^!F$. If $F$ has a right adjoint $R$, then $R^! \dashv F^! \cong F$ and ${}^! R \dashv F$ by the above lemma. Similarly, if $L \dashv F$, then $F \dashv L^!$ and $F \dashv {}^!L$. Summarizing, we have the following result \cite[Lemma 3.5]{MR2869176}:

\begin{lemma}
  \label{lem:ten-cat-adj-2}
  A tensor functor $F$ has a left adjoint if and only if it has a right adjoint. If $L \dashv F \dashv R$, then ${}^! \! R \cong L \cong R^!$ and ${}^! \! L \cong R \cong L^!$.
\end{lemma}

\subsection{The relative modular object}

Now we introduce the notion of the relative modular object for a tensor functor with nice properties.
We first prove the following lemma:

\begin{lemma}
  \label{lem:rel-mod-def}
  Let $F: \mathcal{C} \to \mathcal{D}$ be a tensor functor between tensor categories $\mathcal{C}$ and $\mathcal{D}$, and suppose that it has a left adjoint $L$ and a right adjoint $R$. Then the following assertions are equivalent:
  \begin{itemize}
  \item [(1)] $L$ has a left adjoint.
  \item [(2)] $R$ has a right adjoint.
  \item [(3)] There exists an object $\chi_F \in \mathcal{D}$ such that $R \cong L(- \otimes \chi_F)$.
  \end{itemize}
  Such an object $\chi_F \in \mathcal{D}$ is unique up to isomorphism if it exists. More precisely, if the above conditions hold, then we have isomorphisms
  \begin{equation}
    \label{eq:rel-mod-formula-1}
    E(\unitobj) \cong \chi_F \cong {}^*G(\unitobj), \quad \text{where $E \dashv L \dashv F \dashv R \dashv G$}.
  \end{equation}
  If, moreover, $\mathcal{C}$ and $\mathcal{D}$ are finite tensor categories, then the above three conditions are equivalent to each of the following three conditions:
  \begin{itemize}
  \item [(4)] $L$ is exact.
  \item [(5)] $R$ is exact.
  \item [(6)] $F(P)$ is projective for every projective object $P \in \mathcal{C}$.
  \end{itemize}
\end{lemma}
\begin{proof}
  The equivalence (1) $\Leftrightarrow$ (2) follows from Lemmas~\ref{lem:ten-cat-adj-1} and~\ref{lem:ten-cat-adj-2}. In more detail, if $L$ has a left adjoint $E$, then $E^!$ is right adjoint to $R \cong L^!$. Similarly, if $R$ has a right adjoint $G$, then $G^!$ is left adjoint to $L \cong R^!$.

  To show (2) $\Leftrightarrow$ (3), we note that $\mathcal{D}$ is a left $\mathcal{C}$-module category by the action given by $X \ogreaterthan V = F(X) \otimes V$ ($X \in \mathcal{C}$, $V \in \mathcal{D}$). Suppose that $R$ has a right adjoint $G$, and set $\chi = {}^* G(\unitobj)$. Since $F: \mathcal{C} \to \mathcal{D}$ is a $\mathcal{C}$-module functor, so is $R$ by Lemmas~\ref{lem:C-mod-func-rigid} and \ref{lem:C-mod-func-adj}, and thus so is $G$. Hence, for $X \in \mathcal{C}$, we have
  \begin{equation}
    \label{eq:rel-mod-def-pf-1}
    G(X) = G(X \otimes \unitobj) \cong X \ogreaterthan G(\unitobj) = F(X) \otimes \chi^*.
  \end{equation}
  By definition, $R$ is left adjoint to $G$. On the other hand, by~\eqref{eq:rel-mod-def-pf-1}, we have
  \begin{equation*}
    \Hom_{\mathcal{D}}(V, G(X))
    \cong \Hom_{\mathcal{D}}(V \otimes \chi, F(X))
    \cong \Hom_{\mathcal{C}}(L(V \otimes \chi), X)
  \end{equation*}
  for $V \in \mathcal{D}$ and $X \in \mathcal{C}$. Thus $R \cong L(- \otimes \chi)$. Conversely, if such an object $\chi$ exists, then we see that $F(-) \otimes \chi^*$ is right adjoint to $R$ by a similar computation. Hence we have proved (2) $\Leftrightarrow$ (3).

  Now we suppose that (1), (2) and (3) hold, and let $E$ and $G$ be a left adjoint and a right adjoint of $L$ and $R$, respectively. Then, as we have seen at the beginning of the proof, we have $G^! \cong E$. The second isomorphism in~\eqref{eq:rel-mod-formula-1} has been proved in the above argument. The first one follows from the second one as follows:
  \begin{equation*}
    \chi_F \cong {}^*G(\unitobj) = {}^*G(\unitobj^*) = G^!(\unitobj) \cong E(\unitobj).
  \end{equation*}

  Finally, we assume that $\mathcal{C}$ and $\mathcal{D}$ are finite. Then (1) $\Leftrightarrow$ (4) and (2) $\Leftrightarrow$ (5) follow from Lemma~\ref{lem:EW-thm}. To show that (5) and (6) are equivalent, we make the category $\mathcal{D}$ a finite left $\mathcal{C}$-module category by $F$ in the same way as above. For $V, W \in \mathcal{D}$ and $X \in \mathcal{C}$, there is a natural isomorphism
  \begin{equation}
    \Hom_{\mathcal{D}}(X \ogreaterthan V, W) \cong \Hom_{\mathcal{C}}(X, R(W \otimes V^*)).
  \end{equation}
  Namely, $\iHom(V, W) := R(W \otimes V^*)$ is the internal Hom functor for the $\mathcal{C}$-module category $\mathcal{D}$. By the basic results recalled in \S\ref{subsec:fin-mod-cat}, we see that each of (5) and (6) is equivalent to that $\mathcal{D}$ is an exact $\mathcal{C}$-module category.
\end{proof}

\begin{definition}
  We call $\chi_F^{}$ of Lemma~\ref{lem:rel-mod-def} the {\em relative modular object} of $F$.
\end{definition}

As we explain the detail in Remark~\ref{rem:FMS-thm}, this object generalizes the relative modular function of an extension of finite-dimensional Hopf algebras. The following result is also a motivation of our definition.

\begin{example}
  \label{ex:rel-mod-cent}
  Let $\mathcal{C}$ be a finite tensor category such that $\mathcal{C}^{\env}$ is rigid\footnote{If this assumption is dropped, then $U: \mathcal{Z}(\mathcal{C}) \to \mathcal{C}$ may not satisfy the equivalent conditions of Lemma~\ref{lem:rel-mod-def}. For example, suppose that $k$ is imperfect and set $\mathcal{V} = (\mbox{{\rm mod}-}k, \otimes_k, k)$. Let $\ell$ be a finite inseparable field extension of $k$. Then $\mathcal{C} = (\ell\mbox{-{\rm mod}-}\ell, \otimes_{\ell}, \ell)$ is a finite tensor category over $k$ such that $\mathcal{C}^{\env}$ is not rigid. By Schauenburg \cite{MR1822847}, there are equivalences $\mathcal{Z}(\mathcal{C}) \approx \mathcal{Z}(\mathcal{V}) \approx \mathcal{V}$ of braided monoidal categories. If we identify $\mathcal{Z}(\mathcal{C})$ with $\mathcal{V}$, then the functor $U$ corresponds to the unique (up to isomorphism) tensor functor $\mathcal{V} \to \mathcal{C}$ sending $k \in \mathcal{V}$ to $\ell \in \mathcal{C}$. Thus its right adjoint corresponds to $\Hom_{\ell \otimes_k \ell^{\op}}(\ell, -)$, which is not exact since $\ell$ is inseparable over $k$.}. The main result of \cite{2014arXiv1402.3482S} can be rephrased as follows: The forgetful functor $U: \mathcal{Z}(\mathcal{C}) \to \mathcal{C}$ satisfies the equivalent conditions of Lemma~\ref{lem:rel-mod-def}, and the relative modular object of $U$ is given by $\chi_U^{} = \alpha_{\mathcal{C}}^*$.
\end{example}

Till the end of this subsection, let $F: \mathcal{C} \to \mathcal{D}$ be a tensor functor between tensor categories $\mathcal{C}$ and $\mathcal{D}$ satisfying the equivalent conditions of Lemma~\ref{lem:rel-mod-def}, and let $L$ and $R$ be a left adjoint and a right adjoint of $F$.
We first give the following characterizations of the relative modular object:

\begin{proposition}
  \label{prop:rel-mod-def-2}
  If $\chi = \chi_F$, then there are natural isomorphisms
  \begin{equation}
    \label{eq:rel-mod-1-1}
    L(V \otimes \chi) \cong R(V) \cong L(\chi \otimes V)
    \quad \text{and} \quad
    R(\chi^* \otimes V) \cong L(V) \cong R(V \otimes \chi^*)
  \end{equation}
  for $V \in \mathcal{D}$. Each of these natural isomorphisms characterizes the relative modular object: Namely, if $\chi \in \mathcal{D}$ is an object such that one of the isomorphisms in \eqref{eq:rel-mod-1-1} exists, then $\chi \cong \chi_F^{}$.
\end{proposition}
\begin{proof}
  Let $E$ be a left adjoint of $L$, and let $G$ be a right adjoint of $R$. Taking adjoints, we see that~\eqref{eq:rel-mod-1-1} is equivalent to
  \begin{equation}
    \label{eq:rel-mod-1-2}
    F(-) \otimes \chi^* \cong G \cong {}^*\chi \otimes F(-)
    \quad \text{and} \quad
    \chi^{**} \otimes F(-) \cong E \cong F(-) \otimes \chi.
  \end{equation}
  We now set $\chi = \chi_F$. To prove~\eqref{eq:rel-mod-1-2}, we first remark $\chi^{**} \cong \chi$. Indeed, by Lemma~\ref{lem:ten-cat-adj-1}, we see that ${}^!G$ is left adjoint to $L \cong {}^!R$. Hence, by~\eqref{eq:rel-mod-formula-1}, we have
  \begin{equation}
    \label{eq:rel-mod-1-3}
    \chi \cong E(\unitobj) \cong {}^!G(\unitobj) \cong G(\unitobj)^{*} \cong {}^{*}G(\unitobj)^{**} \cong \chi^{**}.
  \end{equation}
  The category $\mathcal{D}$ is a $\mathcal{C}$-bimodule category via $F$. Since $F$ is a $\mathcal{C}$-bimodule functor, so is $R$, and hence so is $G$. Thus we have natural isomorphisms
  \begin{align}
    \label{eq:rel-mod-1-4}
    F(X) \otimes G(\unitobj) = X \ogreaterthan G(\unitobj)
    \cong G(X)
    \cong G(\unitobj) \olessthan X
    = G(\unitobj) \otimes F(X)
  \end{align}
  for $X \in \mathcal{C}$. Now the first two isomorphisms in~\eqref{eq:rel-mod-1-2} follow from~\eqref{eq:rel-mod-formula-1}, \eqref{eq:rel-mod-1-3} and \eqref{eq:rel-mod-1-4}. The latter two isomorphisms are proved in a similar way by using the fact that $E$ is a $\mathcal{C}$-bimodule functor.

  To see that each of the isomorphisms in \eqref{eq:rel-mod-1-1} characterizes the relative modular object, consider adjoints of them and then use \eqref{eq:rel-mod-formula-1}. For example, if $\mu \in \mathcal{D}$ is an object such that $L \cong R(- \otimes \mu^*)$, then we have $E \cong F(-) \otimes \mu^*$ by taking left adjoints of both sides. Hence $\mu \cong \chi_F$ by \eqref{eq:rel-mod-formula-1}. The other cases are proved analogously.
\end{proof}

\begin{proposition}
  \label{prop:rel-mod-1}
  The object $\chi_F$ has the following properties:
  \begin{enumerate}
  \item $\chi_F \cong \unitobj$ if and only if $F$ is Frobenius.
  \item $\chi_F$ is an invertible object.
  \item $\chi_F$ belongs to the centralizer of $F$.
  \end{enumerate}
\end{proposition}
\begin{proof}
  Write $\chi = \chi_F$. Part (1) is obvious from the definition. By~\eqref{eq:rel-mod-1-1}, we have
  \begin{equation*}
    L(\chi^* \otimes V \otimes \chi)
    \cong R(\chi^* \otimes V)
    \cong L(V)
  \end{equation*}
  for $V \in \mathcal{D}$. Taking right adjoints, we have a natural isomorphism
  \begin{equation*}
    \chi \otimes F(X) \otimes \chi^* \cong F(X)
  \end{equation*}
  for $X \in \mathcal{C}$. Part (2) is proved by letting $X = \unitobj$ (since $\mathcal{D}$ is a multi-ring category in the sense of \cite{MR3242743}, an object $\mu$ of $\mathcal{D}$ is invertible if and only if $\mu \otimes \mu^* \cong \unitobj$; see \cite[\S4.3]{MR3242743}). To prove Part (3), let $E$ be a left adjoint of $L$. In view of \eqref{eq:rel-mod-formula-1}, we may assume $\chi = E(\unitobj)$. As we have seen in the proof of Proposition~\ref{prop:rel-mod-def-2}, $E$ is a $\mathcal{C}$-bimodule functor if we view $\mathcal{D}$ as a $\mathcal{C}$-bimodule category via $F$. Thus,
  \begin{equation*}
    \sigma_V:
    \chi \otimes F(V)
    = E(\unitobj) \olessthan V
    \cong
    E(V)
    \cong
    V \ogreaterthan E(\unitobj)
    = F(V) \otimes \chi
    \quad (V \in \mathcal{C})
  \end{equation*}
  is a natural transformation such that $(\chi, \sigma) \in \mathcal{Z}(F)$.
\end{proof}

\subsection{A formula for the relative modular object}
\label{subsec:rel-mod-obj-formula}

Balan \cite{2014arXiv1411.2236B} proved a similar result as Lemma~\ref{lem:rel-mod-def} from the viewpoint of the theory of Hopf monads. Moreover, Balmer, Dell'Ambrogio and Sanders \cite{2015arXiv150101999B} showed similar results in a quite general (but symmetric) setting of tensor-triangulated categories. Mentioning these results, we wonder that the results of the previous subsection are only an instance of a very general principal in the monoidal category theory.

In any case, our results are not sufficient as a generalization of the theorem of Fischman, Montgomery and Schneider mentioned in Introduction: Their result can be thought of as an explicit formula for the relative modular object of $\mathrm{res}^A_B$ in terms of the modular function (see Remark~\ref{rem:FMS-thm} below), while our results do not give any information about the relative modular object. The second main result of this section is the following formula for the relative modular object:

\begin{theorem}
  \label{thm:mod-obj-3}
  Let $F: \mathcal{C} \to \mathcal{D}$ be an exact tensor functor between finite tensor categories satisfying the equivalent conditions in Lemma~\ref{lem:rel-mod-def}. Then
  \begin{equation*}
    \alpha_{\mathcal{D}}^* \otimes F(\alpha_{\mathcal{C}})
    \cong \chi_F^{}
    \cong F(\alpha_{\mathcal{C}}) \otimes \alpha_{\mathcal{D}}^*.
  \end{equation*}
\end{theorem}
\begin{proof}
  Let $L$ and $R$ be a left and a right adjoint of $F$, respectively. As before, we make $\mathcal{D}$ a left $\mathcal{C}$-module category by $F$. By Lemmas~\ref{lem:C-mod-func-rigid} and~\ref{lem:C-mod-func-adj}, the functor $R$ is a $\mathcal{C}$-module functor. This means that there is a natural isomorphism $R(F(X) \otimes V) \cong X \otimes R(V)$ for $X \in \mathcal{C}$ and $V \in \mathcal{D}$. In other words, there is an isomorphism
  \begin{equation}
    \label{eq:mod-obj-3-pf-1}
    \Upsilon_{\mathcal{C}} \circ R \cong \REX(F, R) \circ \Upsilon_{\mathcal{D}},
  \end{equation}
  where $\Upsilon_{\mathcal{C}}$ and $\Upsilon_{\mathcal{D}}$ are the Cayley functors introduced in \S\ref{subsec:modular-obj}.

  We consider left adjoints of both sides of \eqref{eq:mod-obj-3-pf-1}. As before, we denote by $\Upsilon^*_{\square}$ a left adjoint of $\Upsilon_{\square}$ ($\square = \mathcal{C}, \mathcal{D}$). It is trivial that a left adjoint of $R$ is $F$. To get a left adjoint of $\REX(F, R)$, we note that there are natural isomorphisms
  \begin{equation*}
    \NAT(F \circ T_1, T_2) \cong \NAT(T_1, R \circ T_2)
    \quad \text{and} \quad
    \NAT(T_3 \circ R, T_4) \cong \NAT(T_3, T_4 \circ F)
  \end{equation*}
  for $T_1: \mathcal{C} \to \mathcal{C}$, $T_2: \mathcal{D} \to \mathcal{C}$, $T_3: \mathcal{C} \to \mathcal{D}$ and $T_4: \mathcal{D} \to \mathcal{D}$ (see \cite[X.5 and X.7]{MR1712872} for these natural isomorphisms). Hence we have
  \begin{equation*}
    \NAT(T, R \circ T' \circ F)
    \cong \NAT(F \circ T, T' \circ F)
    \cong \NAT(F \circ T \circ R, F)
  \end{equation*}
  for $T: \mathcal{C} \to \mathcal{C}$ and $T': \mathcal{D} \to \mathcal{D}$. This means $\REX(R, F) \dashv \REX(F, R)$. Thus, taking left adjoints of both sides of \eqref{eq:mod-obj-3-pf-1}, we get
  \begin{equation}
    \label{eq:mod-obj-3-pf-2}
    F \circ \Upsilon_{\mathcal{C}}^* \cong \Upsilon_{\mathcal{D}}^* \circ \REX(R, F).
  \end{equation}

  We now compute the image of $J_{\mathcal{C}} \in \REX(\mathcal{C})$ under \eqref{eq:mod-obj-3-pf-2}. By the definition of the modular object, we have $F(\Upsilon_{\mathcal{C}}^*[J_{\mathcal{C}}]) = F(\alpha_{\mathcal{C}})$.
  On the other hand,
  \begin{align*}
    (\REX(R, F) [J_{\mathcal{C}}]) (V)
    & \cong F(\Hom_{\mathcal{C}}(R(V), \unitobj)^* \bullet \unitobj) \\
    & \cong \Hom_{\mathcal{C}}(L(V \otimes \chi_F), \unitobj)^* \bullet F(\unitobj) \\
    & \cong \Hom_{\mathcal{D}}(V \otimes \chi_F, F(\unitobj))^* \bullet \unitobj \\
    & \cong J_{\mathcal{D}}(V \otimes \chi_F)
  \end{align*}
  for $V \in \mathcal{C}$. If we define an action of $\mathcal{D}$ on $\REX(\mathcal{D})$ by $X \ogreaterthan T = T(- \otimes X)$ for $X \in \mathcal{D}$ and $T \in \REX(\mathcal{D})$, then the above result reads:
  \begin{equation*}
    \REX(R, F)[J_{\mathcal{C}}] \cong \chi_F \ogreaterthan J_{\mathcal{D}}.
  \end{equation*}
  Since $\Upsilon_{\mathcal{D}}: \mathcal{D} \to \REX(\mathcal{D})$ is a left $\mathcal{D}$-module functor, so is $\Upsilon_{\mathcal{D}}^*$. Thus,
  \begin{equation*}
    \Upsilon_{\mathcal{D}}^* \Big[ \REX(R, F)[J_{\mathcal{C}}] \Big]
    \cong \Upsilon_{\mathcal{D}}^* [\chi_F \ogreaterthan J_{\mathcal{D}}]
    \cong \chi_F \otimes \Upsilon_{\mathcal{D}}^* [J_{\mathcal{D}}] = \chi_F \otimes \alpha_{\mathcal{D}}.
  \end{equation*}
  Hence $\chi_F \otimes \alpha_{\mathcal{D}} \cong F(\alpha_{\mathcal{C}})$. Since $\alpha_{\mathcal{D}}$ is invertible, we obtain the first isomorphism of the statement of this theorem. The second isomorphism follows from the first one and the Radford $S^4$-formula.
\end{proof}

\begin{remark}
  \label{rem:FMS-thm}
  We shall explain how this theorem implies a result of Fischman, Montgomery and Schneider \cite{MR1401518}. Let $A/B$ be an extension of finite-dimensional Hopf algebras over $k$, and let $F := \mathrm{res}^A_B: \mbox{\rm mod-}A \to \mbox{\rm mod-}B$ be the restriction functor. As we have mentioned in Introduction, the functors
  \begin{equation*}
    L := (-) \otimes_B A
    \quad \text{and} \quad
    R := (-) \otimes_B \Hom_B(A_B, B_B)
  \end{equation*}
  are a left adjoint and a right adjoint of $F$, respectively. Recall that we have used the Nichols-Zoeller theorem to obtain the above expression of $R$. The theorem allows us to apply our results to $F$. Consequently, we obtain
  \begin{equation}
    \label{eq:FMS-proof-1}
    L(\chi_F \otimes_k X) \cong R(X)
    \quad (X \in \mbox{\rm mod-}B),
  \end{equation}
  where $\chi_F$ is the right $B$-module corresponding to the relative modular function given by~\eqref{eq:rel-mod-func}. The relative Nakayama automorphism $\beta = \beta_{A/B}$ corresponds to the functor $\chi_F \otimes (-)$. Since
  \begin{equation*}
    L(\chi_F \otimes_k X)
    \cong L(X_{\beta})
    \cong (X_{\beta}) \otimes_B A
    \cong X \otimes_B ({}_{\beta^{-1}}A),
  \end{equation*}
  we get an isomorphism ${}_{\beta^{-1}} A_A \cong \Hom_A(A_B, B)$ of $B$-$A$-bimodules from \eqref{eq:FMS-proof-1}. In other words, the extension $A/B$ is $\beta$-Frobenius \cite[Theorem 1.7]{MR1401518}.
\end{remark}

\begin{remark}
  \label{rem:NZ-thm}
  We have used the Nichols-Zoeller theorem to apply Theorem~\ref{thm:mod-obj-3} in the above. Like this, some non-trivial results will be needed to apply our results. Here we note the following criteria: An exact tensor functor $F: \mathcal{C} \to \mathcal{D}$ between finite tensor categories satisfies the equivalent conditions of Lemma~\ref{lem:rel-mod-def} if it is faithful and {\em surjective} in the sense that every object of $\mathcal{D}$ is a quotient of $F(X)$ for some $X \in \mathcal{C}$ \cite[Theorem 2.5 and Section 3]{MR2119143} (notice that our terminology slight differs from theirs).
\end{remark}

\section{Braided Hopf algebras}
\label{sec:braided-Hopf}

\subsection{Main result of this section}

In this section, we give a description of the modular object of the category of right modules over a Hopf algebra in a braided finite tensor category (often called a {\em braided Hopf algebra}). To state our result, we first fix some notations related to Hopf algebras in a braided monoidal category.

Let $\mathcal{V}$ be a braided monoidal category with braiding $\sigma$, and let $H$ be a Hopf algebra with multiplication $m$, unit $u$, comultiplication $\Delta$, counit $\varepsilon$ and antipode $S$. For $M, N \in \mathcal{V}_H$, their tensor product $M \otimes N \in \mathcal{V}$ is a right $H$-module by
\begin{equation}
  \label{eq:H-mod-tensor}
  \triangleleft_{M \otimes N} = (\triangleleft_M \otimes \triangleleft_N) \circ (\id_M \otimes \sigma_{N,H} \otimes \id_H) \circ (\id_M \otimes \id_N \otimes \Delta),
\end{equation}
where $\triangleleft_M$ and $\triangleleft_N$ are the actions of $H$ on $M$ and $N$, respectively. The category $\mathcal{V}_H$ is a monoidal category with this operation. In a similar way, the category ${}_H \mathcal{V}$ is also a monoidal category. We note that $\mathcal{V}_H$ and ${}_H \mathcal{V}$ are rigid if $\mathcal{V}$ is rigid and $S$ is invertible.

We recall basic results on integrals of a braided Hopf algebra:
Suppose that $\mathcal{V}$ is rigid and admits equalizers.
Then the antipode of $H$ is invertible. An ($X$-based) {\em right integral in $H$} is a morphism $t: X \to H$ in $\mathcal{V}$ such that $m \circ (t \otimes H) = t \otimes \varepsilon$. The category of right integrals in $H$ (defined as the full subcategory of the category of objects over $H$ \cite[II.6]{MR1712872}) has a terminal object. We write it as $\Lambda: \mathrm{Int}(H) \to H$ and call $\mathrm{Int}(H) \in \mathcal{V}$ the {\em object of integrals} in $H$. It is known that $\mathrm{Int}(H)$ is an invertible object.
See \cite{MR1685417} and \cite{MR1759389} for the above results.

\begin{definition}
  The (right) {\em modular function} on $H$ is a morphism $\alpha_H: H \to \unitobj$ of algebras in $\mathcal{V}$ determined by the following equation:
  \begin{equation}
    \label{eq:br-Hopf-modular-ft}
    \alpha_H \otimes \id_{\mathrm{Int}(H)} = m \circ (\id_H \otimes \Lambda).
  \end{equation}
\end{definition}

We regard an object $V \in \mathcal{V}$ as a right $H$-module by defining the action of $H$ by the counit of $H$. We also identify an algebra morphism $\gamma: H \to \unitobj$ with the right $H$-module whose underlying object is $\unitobj \in \mathcal{V}$ and whose action is given by $\gamma$. With the above notation, the main result of this section is stated as follows:

\begin{theorem}
  \label{thm:mod-obj-br-Hopf}
  Let $\mathcal{V}$ be a braided finite tensor category over $k$, and let $H$ be a Hopf algebra in $\mathcal{V}$. The modular object of $\mathcal{C} = \mathcal{V}_H$ is given by
  \begin{equation*}
    \alpha_{\mathcal{C}} = \mathrm{Int}(H)^* \otimes \alpha_{\mathcal{V}} \otimes \alpha_H.
  \end{equation*}
\end{theorem}

The order of the tensorands in the right-hand side of the above equation is arbitrary, since they commute with each other up to isomorphism.

Since the antipode $S: H \to H$ is an anti-algebra isomorphism, $\mathcal{V}_H$ and ${}_H \mathcal{V}$ are equivalent. This equivalence is not monoidal in general, but preserves the unit object. Thus, by Corollary~\ref{cor:mod-obj-invariance} and the above theorem, we have the following description of the modular object of the category of left $H$-modules:

\begin{corollary}
  \label{cor:mod-obj-br-Hopf}
  The modular object of $\mathcal{C} = {}_H \mathcal{V}$ is given by
  \begin{equation*}
    \alpha_{\mathcal{C}} = \mathrm{Int}(H)^* \otimes \alpha_{\mathcal{V}} \otimes \overline{\alpha}_H,
  \end{equation*}
  where $\overline{\alpha}_H$ is the left modular function given by $\overline{\alpha}_H = \alpha_H \circ S$.
\end{corollary}

By Theorem~\ref{thm:mod-obj-br-Hopf} and Corollary~\ref{cor:mod-obj-br-Hopf}, we obtain:

\begin{corollary}
  For a Hopf algebra $H$ in $\mathcal{V}$, the following are equivalent:
  \begin{enumerate}
  \item The category of left $H$-modules in $\mathcal{V}$ is unimodular.
  \item The category of right $H$-modules in $\mathcal{V}$ is unimodular.
  \item $\mathrm{Int}(H) \cong \alpha_{\mathcal{V}}$ and $\alpha_{H} = \varepsilon$.
  \end{enumerate}
\end{corollary}

By an extension $A/B$ of Hopf algebras in a braided monoidal category $\mathcal{V}$, we mean a morphism $i_{A/B}: B \to A$ of Hopf algebras in $\mathcal{V}$ being monic as a morphism in $\mathcal{V}$. As in the ordinary case, the functor $\mathcal{V}_A \to \mathcal{V}_B$ induced by $i_{A/B}$ is called the {\em restriction functor}. As an application of Theorems~\ref{thm:mod-obj-3} and~\ref{thm:mod-obj-br-Hopf}, we will  show the following theorem:

\begin{theorem}
  \label{thm:FMS-braided}
  Let $\mathcal{V}$ be a braided finite tensor category over $k$. For an extension $A/B$ of Hopf algebras in $\mathcal{V}$, the following assertions are equivalent:
  \begin{enumerate}
  \item The restriction functor $\mathcal{V}_A \to \mathcal{V}_B$ is a Frobenius functor.
  \item $\mathrm{Int}(A) \cong \mathrm{Int}(B)$ and $\alpha_A \circ i_{A/B} = \alpha_B$.
  \end{enumerate}
\end{theorem}

The rest of this section is devoted to the proofs of Theorems~\ref{thm:mod-obj-br-Hopf} and~\ref{thm:FMS-braided}.

\subsection{Nakayama automorphism}

Let $\mathcal{V}$ be a braided rigid monoidal category with braiding $\sigma$, and suppose that it admits equalizers. We use the graphical techniques to express morphisms in $\mathcal{V}$. Our convention is that a morphism goes from the top of the diagram to the bottom ({\it cf}. \cite{MR1685417}). Following, the braiding $\sigma$, its inverse, the evaluation $\eval: X^* \otimes X \to \unitobj$ and the coevaluation $\coev: \unitobj \to X \otimes X^*$ are expressed as follows:
\begin{equation*}
  \sigma_{X,Y} = \begin{array}{c} \includegraphics{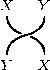} \end{array} \quad
  \sigma_{X,Y}^{-1} = \begin{array}{c} \includegraphics{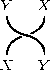} \end{array} \quad
  \eval = \begin{array}{c} \includegraphics{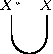} \end{array} \quad
  \coev = \begin{array}{c} \includegraphics{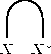} \end{array}
\end{equation*}
The axioms for Hopf algebras in $\mathcal{V}$ are expressed as in Figure~\ref{fig:br-Hopf-axioms}. Here, for a Hopf algebra $H$ in $\mathcal{V}$, we depict its structure morphisms as follows:
\begin{equation*}
  m = \begin{array}{c} \includegraphics{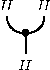} \end{array} \quad
  u = \begin{array}{c} \includegraphics{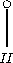} \end{array} \quad
  \Delta = \begin{array}{c} \includegraphics{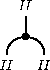} \end{array} \quad
  \varepsilon = \begin{array}{c} \includegraphics{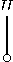} \end{array} \quad
  S = \begin{array}{c} \includegraphics{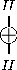} \end{array} \quad
  S^{-1} = \begin{array}{c} \includegraphics{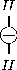} \end{array}
\end{equation*}

\begin{figure}
  \centering
  \includegraphics{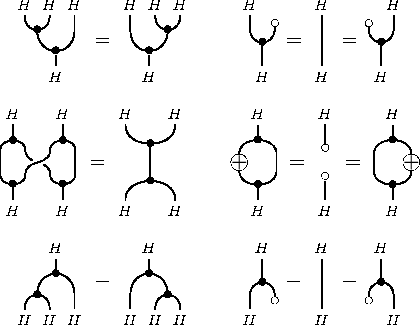}
  \caption{The axioms of a braided Hopf algebra}
  \label{fig:br-Hopf-axioms}
\end{figure}

As in the previous subsection, we fix a terminal object $\Lambda: \mathrm{Int}(H) \to H$ of the category of right integrals in $H$. By definition, we have
\begin{equation}
  \label{eq:br-Hopf-r-int-1}
  m \circ (\Lambda \otimes \id_H) = \Lambda \otimes \varepsilon.
\end{equation}
It is known that there exists a unique morphism $\lambda: H \to \mathrm{Int}(H)$ such that
\begin{equation}
  \label{eq:br-Hopf-r-int-2}
  (\lambda \otimes \id_H) \circ \Delta = \lambda \otimes u
  \text{\quad and \quad}
  \lambda \circ \Lambda = \id_{\mathrm{Int}(H)}.
\end{equation}
The paring $\lambda \circ m: H \otimes H \to \mathrm{Int}(H)$ is non-degenerate in the sense that
\begin{equation}
  \label{eq:br-Hopf-pairing}
  \phi_H: H \xrightarrow{\quad H \otimes \coev \quad} H \otimes H \otimes H^*
  \xrightarrow{\quad (\lambda \circ m) \otimes H^* \quad} \mathrm{Int}(H) \otimes H^*
\end{equation}
is invertible (see \cite{MR1685417} and \cite{MR1759389} for these results).

\begin{definition}[Doi-Takeuchi {\cite{MR1800707}}]
  The {\em Nakayama automorphism} of $H$ is the unique morphism $\mathcal{N}: H \to H$ in $\mathcal{V}$ such that
  \begin{equation}
    \label{eq:Nak-auto-def}
    \lambda \circ m \circ \sigma_{H,H} = \lambda \circ m \circ (\id_H \otimes \mathcal{N}).
  \end{equation}
\end{definition}

Let $K \in \mathcal{V}$ be an invertible object. The {\em monodromy around $K$} is the natural transformation $\Omega(K): \id_{\mathcal{V}} \to \id_{\mathcal{V}}$ defined by
\begin{equation}
  \label{eq:monodromy-1}
  \id_K \otimes \Omega(K)_V = \sigma_{V,K} \circ \sigma_{K,V}
\end{equation}
for $V \in \mathcal{V}$. The definition of a braiding implies
\begin{equation}
  \label{eq:monodromy-2}
  \Omega(\unitobj) = \id,
  \quad \Omega(K \otimes K') = \Omega(K) \circ \Omega(K')
  \text{\quad and \quad} \Omega(K^*) = \Omega(K)^{-1}
\end{equation}
for invertible objects $K$ and $K'$.

\begin{lemma}
  \label{lem:Nak-auto}
  The Nakayama automorphism of $H$ is given by
  \begin{equation}
    \label{eq:Nak-auto-DT}
    \mathcal{N} = S^{-2} \circ (\alpha_H \otimes \id_H) \circ \Delta \circ \Omega(\mathrm{Int}(H))_H
  \end{equation}
\end{lemma}

This formula has been given by Doi and Takeuchi under the assumption that $\mathcal{V}$ is, in a sense, built on the category of vector spaces \cite[Proposition 13.1]{MR1800707}. Since their argument cannot be applied to our general setting, we give a proof.

\begin{proof}
  Set $I = \mathrm{Int}(H)$, $\alpha = \alpha_H$ and $\omega = \Omega(I)_H$. Then~\eqref{eq:Nak-auto-DT} is equivalent to
  \begin{equation}
    \label{eq:Nak-auto-DT-2}
    S^2 \circ \mathcal{N} = (\alpha \otimes \id_H) \circ \Delta \circ \omega.
  \end{equation}
  To prove this equation, we first prove
  \begin{equation}
    \label{eq:lem-Nak-auto-pf-1}
    (m \otimes \id_H) \circ (\id_H \otimes \sigma_{H,H}) \circ (\Delta \otimes \id_H)
    = (\id_H \otimes m) \circ (\Delta m \otimes S) \circ (\id_H \otimes \Delta)
  \end{equation}
  as in Figure~\ref{fig:Nak-auto-Fig-1} (this equation reads ``$a_{(1)} b \otimes a_{(2)} = (a b_{(1)})_{(1)} \otimes (a b_{(1)})_{(2)} S(b_{(2)})$'' in an ordinary Hopf algebra in the Sweedler notation). Using \eqref{eq:br-Hopf-r-int-1}, \eqref{eq:br-Hopf-r-int-2} and~\eqref{eq:lem-Nak-auto-pf-1}, we obtain three formulas depicted in Figure~\ref{fig:Nak-auto-Fig-2}. Now \eqref{eq:Nak-auto-DT-2} is proved as in Figure~\ref{fig:Nak-auto-Fig-3}.
\end{proof}

\begin{figure}
  \centering
  \includegraphics{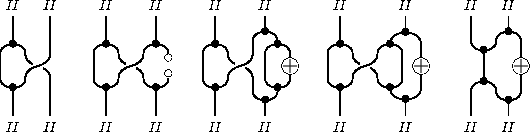}
  \caption{Graphical proof of Equation~\eqref{eq:lem-Nak-auto-pf-1}}
  \label{fig:Nak-auto-Fig-1}

  \bigskip
  \includegraphics{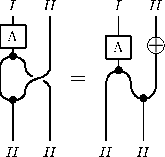}
  \qquad \includegraphics{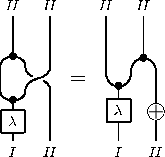}
  \qquad \includegraphics{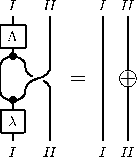}
  \caption{Formulas involving integrals}
  \label{fig:Nak-auto-Fig-2}

  \bigskip
  \begin{align*}
    \begin{array}{c} \includegraphics{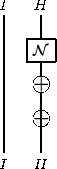} \end{array}
    & \overset{\text{Figure~\ref{fig:Nak-auto-Fig-2}}}{=}
    \begin{array}{c} \includegraphics{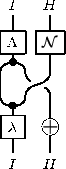} \end{array}
    \overset{\eqref{eq:Nak-auto-def}}{=}
    \begin{array}{c} \includegraphics{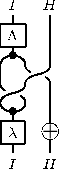} \end{array}
    {\ = \ }
    \begin{array}{c} \includegraphics{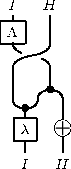} \end{array} \\
    & \overset{\text{Figure~\ref{fig:Nak-auto-Fig-2}}}{=}
    \begin{array}{c} \includegraphics{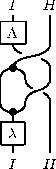} \end{array}
    \overset{\eqref{eq:br-Hopf-modular-ft}}{=}
    \begin{array}{c} \includegraphics{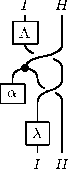} \end{array}
    \overset{\eqref{eq:br-Hopf-r-int-2}}{=}
    \ \begin{array}{c} \includegraphics{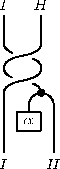} \end{array}
    \overset{\eqref{eq:monodromy-1}}{=}
    \begin{array}{c} \includegraphics{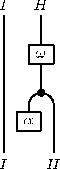} \end{array}    
  \end{align*}
  \caption{Graphical proof of Equation~\eqref{eq:Nak-auto-DT-2}}
  \label{fig:Nak-auto-Fig-3}
\end{figure}

\subsection{The fundamental theorem for Hopf bimodules}

Let $\mathcal{V}$ and $H$ be as in the previous subsection. For an $H$-bimodule $M \in {}_H \mathcal{V}_H$, we express the left action $\triangleright_M$ and the right action $\triangleleft_M$ of $H$ on $M$ respectively as follows:
\begin{equation*}
  \triangleright_M =
  \begin{array}{c} \includegraphics{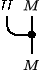} \end{array}
  \text{\quad and \quad}
  \triangleleft_M =
  \begin{array}{c} \includegraphics{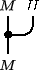} \end{array}
\end{equation*}
Given $M, N \in {}_H \mathcal{V}_H$, we define $M \barotimes N$ to be the $H$-bimodule with the underlying object $M \otimes N$ and with actions
\begin{equation}
  \label{eq:bimod-tensor}
  \triangleright_{M \barotimes N} =
  \begin{array}{c} \includegraphics{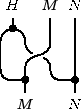} \end{array}
  \text{\quad and \quad}
  \triangleleft_{M \barotimes N} =
  \begin{array}{c} \includegraphics{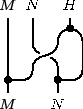} \end{array}
\end{equation}
The category ${}_H \mathcal{V}_H$ is a rigid monoidal category with respect to $\barotimes$. Note that the left dual object of $M \in {}_H \mathcal{V}_H$ (with respect to $\barotimes$), denoted by $M^{\vee}$, is the $H$-bimodule with underlying object $M^*$ and with actions
\begin{equation}
  \label{eq:bimod-dual}
  \triangleright_{M^{\vee}} =
  \begin{array}{c} \includegraphics{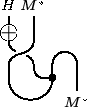} \end{array}
  \text{\quad and \quad}
  \triangleleft_{M^{\vee}} =
  \begin{array}{c} \includegraphics{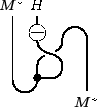} \end{array}
\end{equation}

The object $H \in \mathcal{V}$ is an $H$-bimodule by $m$. Moreover, the coalgebra $(H, \Delta, \varepsilon)$ in $\mathcal{V}$ is in fact a coalgebra in $({}_H \mathcal{V}_H, \barotimes, \unitobj)$. We denote by ${}^H_H \mathcal{V}^{}_H$ the category of left $H$-comodules in ${}_H \mathcal{V}_H$ and refer to an object of this category as a {\em Hopf $H$-bimodule}. By definition, an object of ${}^H_H \mathcal{V}_H^{}$ is an $H$-bimodule $M$ in $\mathcal{V}$ endowed with a left $H$-comodule structure $\delta_M: M \to H \otimes M$ in $\mathcal{V}$ such that
\begin{equation}
  \label{eq:Hopf-bimod-cond}
  \begin{array}{c} \includegraphics{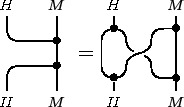} \end{array}
  \text{\quad and \quad}
  \begin{array}{c} \includegraphics{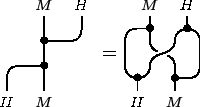} \end{array}
\end{equation}
(here, the coaction $\delta_M$ is expressed by the picture of $\triangleright_M$ turned upside down). Thus a Hopf $H$-bimodule defined here is the same thing as a two-fold Hopf bimodule of Bespalov and Drabant \cite[{Definition 3.6.1}]{MR1492897}.

For $M \in {}_H^H\mathcal{V}_H^{}$, we set $\pi_M := \triangleright_M \circ (S \otimes \id_M) \circ \delta_M: M\to M$. The {\em coinvariant} of $M$, denoted by $M^{\mathrm{co}H}$, is defined to be the equalizer of $\pi_M$ and $\id_M$. It is also an equalizer of $\delta_M$ and $u \otimes \id_M$. Thus, symbolically, we have
\begin{equation}
  \label{eq:Hopf-bimod-coinv}
  M^{\mathrm{co}H}
  = \mathrm{Eq}(\pi_M, \id_M)
  = \mathrm{Eq}(\delta_M, u \otimes \id_M).
\end{equation}
Let $M_{\mathrm{ad}}$ be the right $H$-module with underlying object $M$ and with action
\begin{equation}
  \label{eq:Hopf-bimod-adj}
  \triangleleft_M^{\mathrm{ad}}
  = \triangleright_M \circ (S \otimes \triangleleft_M) \circ (\sigma_{M,H} \otimes \id_H) \circ (\id_M \otimes \Delta).
\end{equation}
We call $\triangleleft_M^{\mathrm{ad}}$ the {\em adjoint action} and express it by the same diagram as a right action but with labeled `ad' (as in the first diagram in Figure~\ref{fig:H-dual-adj-act}). The morphism $\pi_M$ is in fact an $H$-linear idempotent on the right $H$-module $M_{\mathrm{ad}}$. Hence $M^{\mathrm{co}H}$ is a right $H$-module as a submodule of $M_{\mathrm{ad}}$.

An object of the form $H \barotimes M$, $M \in {}_H \mathcal{V}_H$, is a Hopf $H$-bimodule as a free left $H$-comodule in ${}_H \mathcal{V}_H$. We regard a right $H$-module as an $H$-bimodule by defining the left action by $\varepsilon$. The {\em fundamental theorem for Hopf bimodules} (Bespalov and Drabant \cite[Proposition 3.6.3]{MR1492897}) states that the functor
\begin{equation}
  \label{eq:Hopf-bimod-equiv}
  H \barotimes (-): \mathcal{V}_H^{}
  \to {}^H_H \mathcal{V}^{}_H,
  \quad V \mapsto H \barotimes V
  \quad (V \in \mathcal{V}_H)
\end{equation}
is an equivalence of categories with a quasi-inverse $(-)^{\mathrm{co}H}$.

The left dual object of a right $H$-comodule in ${}_H \mathcal{V}_H$ is a left $H$-comodule in a natural way. We are interested in the Hopf bimodule $H^{\vee} \in {}^H_H \mathcal{V}_H^{}$ (where $H$ is regarded as a right $H$-comodule in ${}_H \mathcal{V}_H$ by the comultiplication). In view of the fundamental theorem, it is essential to determine its coinvariant.

\begin{lemma}
  \label{lem:Hopf-bim-dual}
  $(H^{\vee})^{\mathrm{co}H} \cong \mathrm{Int}(H)^* \otimes \alpha_H$ as right $H$-modules.
\end{lemma}
\begin{proof}
  As remarked in \cite[\S3.1]{MR1759389}, $\lambda$ is a coequalizer of
  \begin{equation*}
    f := \id_H \otimes {}^* u
    \text{\quad and \quad} g := (\id_H \otimes \eval) \circ (\Delta \otimes \id_{{}^* \! H}).
  \end{equation*}
  Note that $g^*$ is the coaction of $H$ on $H^{\vee} \in {}_H^H \mathcal{V}_H^{}$. By \eqref{eq:Hopf-bimod-coinv}, $\lambda^*: \mathrm{Int}(H)^* \to H^*$ is the coinvariant of $H^{\vee}$. Hence the claim of this lemma is equivalent to
  \begin{equation}
    \label{eq:Hopf-bimod-dual-pf-1}
    \triangleleft_{H^{\vee}}^{\mathrm{ad}} \circ (\lambda^* \otimes \id_H) = \lambda^* \otimes \alpha_H
  \end{equation}
  For simplicity, we set $I = \mathrm{Int}(H)$, $\alpha = \alpha_H$ and $\overline{\omega} = \Omega(I)_H^{-1}$. The adjoint action is computed as in Figure~\ref{fig:H-dual-adj-act}. The equation equivalent to \eqref{eq:Hopf-bimod-dual-pf-1} via
  \begin{equation*}
    \Hom_{\mathcal{V}}(I^* \otimes H, H^*) \cong \Hom_{\mathcal{V}}(H \otimes H, I)
  \end{equation*}
  is proved as in Figure~\ref{fig:H-dual-coinv}. 
\end{proof}

\begin{figure}
  \centering
  \begin{equation*}
    \begin{array}{c} \includegraphics{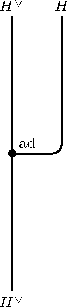} \end{array}
    \overset{\eqref{eq:Hopf-bimod-adj}}{=}
    \begin{array}{c} \includegraphics{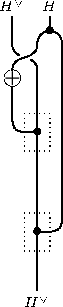} \end{array}
    \overset{\eqref{eq:bimod-dual}}{=}
    \begin{array}{c} \includegraphics{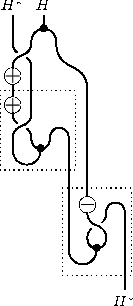} \end{array}
    {\ = \ }
    \begin{array}{c} \includegraphics{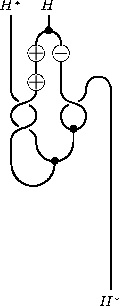} \end{array}
  \end{equation*}
  \caption{The computation of the adjoint action}
  \label{fig:H-dual-adj-act}

  \bigskip
  \begin{equation*}
    \begin{array}{c} \includegraphics{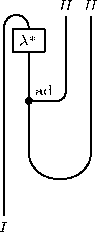} \end{array}
    \overset{\text{Figure~\ref{fig:H-dual-adj-act}}}{=}
    \begin{array}{c} \includegraphics{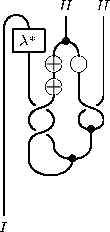} \end{array}
    {\ = \ }
    \begin{array}{c} \includegraphics{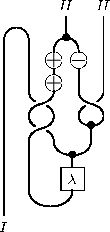} \end{array}
  \end{equation*}
  \begin{equation*}
    \overset{\eqref{eq:monodromy-1}, \eqref{eq:monodromy-2}}{=}
    \begin{array}{c} \includegraphics{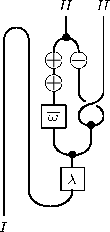} \end{array}
    \overset{\eqref{eq:Nak-auto-def}}{=}
    \begin{array}{c} \includegraphics{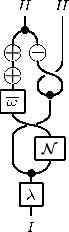} \end{array}
    \overset{\eqref{eq:Nak-auto-DT}}{=}
    \begin{array}{c} \includegraphics{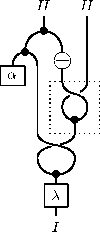} \end{array}
  \end{equation*}
  \begin{equation*}
    {\ = \ } \begin{array}{c} \includegraphics{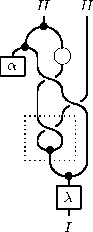} \end{array}
    {\ = \ } \begin{array}{c} \includegraphics{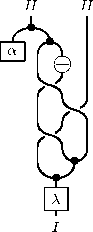} \end{array}
    {\ = \ } \begin{array}{c} \includegraphics{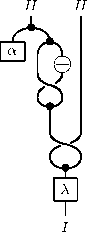} \end{array}
    {\ = \ } \begin{array}{c} \includegraphics{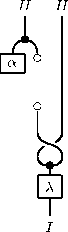} \end{array}
    {\ = \ } \begin{array}{c} \includegraphics{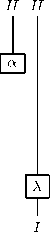} \end{array}
  \end{equation*}
  \caption{The computation of the action of $H$ on $(H^{\vee})^{\mathrm{co}H}$}
  \label{fig:H-dual-coinv}
\end{figure}

\subsection{Proof of Theorem~\ref{thm:mod-obj-br-Hopf}}

We now prove Theorem~\ref{thm:mod-obj-br-Hopf}. We recall the assumptions: $\mathcal{V}$ is a braided finite tensor category over $k$ with braiding $\sigma$, and $H$ is a Hopf algebra in $\mathcal{V}$. We regard $\mathcal{V}$ as a full subcategory of $\mathcal{C} := \mathcal{V}_H$ by regarding an object of $\mathcal{V}$ as a right $H$-module by the counit of $H$. There are forgetful functors
\begin{equation*}
  \REX_{\mathcal{C}}(\mathcal{C})
  \xrightarrow{\ \Theta_{\mathcal{C}} \ } \REX(\mathcal{C})
  \text{\quad and \quad}
  \REX_{\mathcal{C}}(\mathcal{C})
  \xrightarrow{\ \Theta_{\mathcal{C}/\mathcal{V}} \ }
  \REX_{\mathcal{V}}(\mathcal{C})
  \xrightarrow{\ \Theta_{\mathcal{V}} \ } \REX(\mathcal{C})
\end{equation*}
such that $\Theta_{\mathcal{C}} = \Theta_{\mathcal{C}/\mathcal{V}} \circ \Theta_{\mathcal{V}}$. Since the Cayley functor $\Upsilon_{\mathcal{C}}$ corresponds to $\Theta_{\mathcal{C}}$ under the identification $\REX_{\mathcal{C}}(\mathcal{C}) \approx \mathcal{C}$, the composition
\begin{equation*}
  \Theta_{\mathcal{C}}^*: \REX(\mathcal{C})
  \xrightarrow{\quad \Upsilon_{\mathcal{C}}^* \quad} \mathcal{C}
  \xrightarrow{\quad \approx \quad}
  \REX_{\mathcal{C}}(\mathcal{C})
\end{equation*}
is left adjoint to $\Theta_{\mathcal{C}}$. Now we set $J_{\mathcal{C}} = \Hom_{\mathcal{C}}(-, \unitobj)^* \bullet \unitobj \in \REX(\mathcal{C})$. By the definition of the modular object, we have:

\begin{lemma}
  \label{lem:mod-obj-br-Hopf-1}
  $\Theta_{\mathcal{C}}^* [J_{\mathcal{C}}] \cong (-) \otimes \alpha_{\mathcal{C}}$ in $\REX_{\mathcal{C}}(\mathcal{C})$.
\end{lemma}

Our main idea of the proof of Theorem~\ref{thm:mod-obj-br-Hopf} is to compute $\Theta_{\mathcal{C}}^* [J_{\mathcal{C}}]$ in terms of left adjoints of $\Theta_{\mathcal{V}}$ and $\Theta_{\mathcal{C}/\mathcal{V}}$. We first prove:

\begin{lemma}
  \label{lem:mod-obj-br-Hopf-2}
  The functor $\Theta_{\mathcal{V}}$ has a left adjoint, say $\Theta_{\mathcal{V}}^*$. We have
  \begin{equation*}
    \Theta_{\mathcal{V}}^*[J_{\mathcal{C}}] \cong (-) \otimes_H \alpha_{\mathcal{V}},
  \end{equation*}
  where $\alpha_{\mathcal{V}} \in \mathcal{V}$ is regarded as an $H$-bimodule by the counit of $H$.
\end{lemma}
\begin{proof}
  That $\Theta_{\mathcal{V}}$ has a left adjoint follows from Lemma~\ref{lem:Rex-C-as-modules}. To prove the claim, we note that the inclusion functor $i: \mathcal{V} \to \mathcal{C}$ has a left adjoint
  \begin{equation*}
    t: \mathcal{C} \to \mathcal{V}, \quad X \mapsto X \otimes_H \unitobj.
  \end{equation*}
  Since $i$ and $t$ are $k$-linear right exact $\mathcal{V}$-module functors, 
  by Lemma~\ref{lem:Rex-F-G}, they induce a $k$-linear right exact strong $\mathcal{V}^{\env}$-module functor
  \begin{equation*}
    \Omega := \REX(t, i): \REX(\mathcal{V}) \to \REX(\mathcal{C}),
    \quad F \mapsto i \circ F \circ t.
  \end{equation*}
  Now let $A = \int^{X \in \mathcal{V}} X \boxtimes X^*$ be the algebra in $\mathcal{V}^{\env}$ considered in Lemma~\ref{lem:coend-algebra}. Since $\Omega$ is a strong $\mathcal{V}^{\env}$-module functor, it induces a functor (also denoted by $\Omega$) between the categories of $A$-modules in such a way that the following diagram commutes up to isomorphism:
  \begin{equation*}
    \xymatrix{
      \REX(\mathcal{V})
      \ar[d]_{\Omega}
      \ar[rr]^{A \ogreaterthan (-)}
      & & {}_A \REX(\mathcal{V}) \ar[d]_{\Omega}
      \ar[rr]^{\text{Lemma~\ref{lem:Rex-C-as-modules}}}_{\cong}
      & & \REX_{\mathcal{V}}(\mathcal{V})
      \ar@{..>}[d]^{\Omega} \\
      \REX(\mathcal{C})
      \ar[rr]^{A \ogreaterthan (-)}
      & & {}_A \REX(\mathcal{C})
      \ar[rr]^{\text{Lemma~\ref{lem:Rex-C-as-modules}}}_{\cong}
      & & \REX_{\mathcal{V}}(\mathcal{C})
    }
  \end{equation*}
  Now we chase $J_{\mathcal{V}} := \Hom_{\mathcal{V}}(\unitobj, -)^* \bullet \unitobj \in \REX(\mathcal{V})$ around this diagram. Since the composition along the bottom row is $\Theta_{\mathcal{V}}^*$, we have
  \begin{equation*}
    \Theta_{\mathcal{V}}^* \Omega[ J_{\mathcal{V}}] \cong \Omega \big[ A \ogreaterthan J_{\mathcal{V}} \big]
  \end{equation*}
  in $\REX_{\mathcal{V}}(\mathcal{C}) \cong {}_A \REX(\mathcal{C})$. Since $t$ is left adjoint to $i$, we have
  \begin{align*}
    \Omega[J_{\mathcal{V}}]
    = \Hom_{\mathcal{V}}(t(-), \unitobj)^* \bullet \unitobj
    \cong \Hom_{\mathcal{C}}(-, \unitobj)^* \bullet \unitobj = J_{\mathcal{C}}.
  \end{align*}
  Since $F \mapsto A \ogreaterthan F$ is left adjoint to the forgetful functor ${}_A \REX(\mathcal{V}) \to \REX(\mathcal{V})$, we have $A \ogreaterthan J_{\mathcal{V}} \cong (-) \otimes \alpha_{\mathcal{V}}$ by the definition of the modular object. Hence we have
  \begin{equation*}
    \Omega \big[ A \ogreaterthan J_{\mathcal{V}} \big]
    \cong \Omega \big[ (-) \ogreaterthan \alpha_{\mathcal{V}} \big]
    \cong t(-) \otimes \alpha_{\mathcal{V}}
    \cong (-) \otimes_H \alpha_{\mathcal{V}}
  \end{equation*}
  in $\REX_{\mathcal{V}}(\mathcal{C})$. Therefore $\Theta_{\mathcal{V}}^*[J_{\mathcal{C}}] \cong (-) \otimes_H \alpha_{\mathcal{V}}$.
\end{proof}

We now consider the forgetful functor $\Theta_{\mathcal{C}/\mathcal{V}}: \REX_{\mathcal{C}}(\mathcal{C}) \to \REX_{\mathcal{V}}(\mathcal{C})$. 

\begin{lemma}
  \label{lem:mod-obj-br-Hopf-3}
  $\Theta_{\mathcal{C}/\mathcal{V}}$ has a left adjoint, say $\Theta_{\mathcal{C}/\mathcal{V}}^*$.
  For $M \in {}_H \mathcal{V}_H$, we have
  \begin{equation*}
    \Theta_{\mathcal{C}/\mathcal{V}}^* \Big[ (-) \otimes_H M \Big]
    \cong (-) \otimes (H^{\vee} \barotimes M)^{\mathrm{co}H}.
  \end{equation*}
\end{lemma}
\begin{proof}
  We consider the following diagram:
  \begin{equation*}
    \xymatrix{
      \REX_{\mathcal{C}}(\mathcal{C})
      \ar[rrrr]^{\Theta_{\mathcal{C}/\mathcal{V}}}
      \ar[d]_{\approx}
      & & & & \REX_{\mathcal{V}}(\mathcal{C})
      \ar[d]^{\approx} \\
      \mathcal{V}_H
      \ar[rr]^{H \barotimes (-)}
      & & {}^H_H \mathcal{V}_H^{}
      \ar[rr]^{F}
      & & {}_H \mathcal{V}_H,
    }
  \end{equation*}
  where the vertical arrows are the equivalences given by Lemma~\ref{lem:EW-thm} and $F$ is the functor forgetting the comodule structure. Since there is an isomorphism
  \begin{equation*}
    X \otimes_H (H \barotimes V) \cong X \otimes V
    \quad (X, V \in \mathcal{C}),
  \end{equation*}
  the diagram commutes up to isomorphisms. The functors in the diagram are equivalences except $F$ and $\Theta_{\mathcal{C}/\mathcal{V}}$, and the functor tensoring $H^{\vee}$ is left adjoint to $F$. Thus $\Theta_{\mathcal{C}/\mathcal{V}}$ has a left adjoint as the composition of functors having left adjoints. Hence we get the following diagram commuting up to isomorphism:
  \begin{equation*}
    \xymatrix{
      \REX_{\mathcal{C}}(\mathcal{C})
      \ar@{<-}[rrrr]^{\Theta_{\mathcal{C}/\mathcal{V}}^*}
      \ar@{<-}[d]_{\approx}
      & & & & \REX_{\mathcal{V}}(\mathcal{C})
      \ar@{<-}[d]^{\approx} \\
      \mathcal{V}_H
      \ar@{<-}[rr]^{(-)^{\mathrm{co}H}}
      & & {}^H_H \mathcal{V}_H^{}
      \ar@{<-}[rr]^{H^{\vee} \barotimes (-)}
      & & {}_H \mathcal{V}_H.
    }
  \end{equation*}
  Now the claim is obtained by chasing $M \in {}_H \mathcal{V}_H$ around this diagram.
\end{proof}

\begin{proof}[Proof of Theorem~\ref{thm:mod-obj-br-Hopf}]
  Since $\Theta_{\mathcal{C}} = \Theta_{\mathcal{V}} \circ \Theta_{\mathcal{C}/\mathcal{V}}$, we have
  \begin{equation}
    \label{eq:mod-obj-br-Hopf-1}
    \Theta_{\mathcal{C}}^*[J_{\mathcal{C}}] = (\Theta_{\mathcal{C}/\mathcal{V}}^* \circ \Theta_{\mathcal{V}}^*)[J_{\mathcal{C}}].
  \end{equation}
  The left-hand side is $(-) \otimes \alpha_{\mathcal{C}}$ by Lemma~\ref{lem:mod-obj-br-Hopf-1}. Set $I = \mathrm{Int}(H)$. The right-hand side is computed as follows:
  \begin{align*}
    (\Theta_{\mathcal{C}/\mathcal{V}}^* \circ \Theta_{\mathcal{V}}^*)[J_{\mathcal{C}}]
    & \cong \Theta_{\mathcal{C}/\mathcal{V}}^* \Big[ (-) \otimes_H \alpha_{\mathcal{V}} \Big]
    & & \text{(by Lemma~\ref{lem:mod-obj-br-Hopf-2})} \\
    & \cong (-) \otimes (H^{\vee} \barotimes \alpha_{\mathcal{V}})^{\mathrm{co}H}
    & & \text{(by Lemma~\ref{lem:mod-obj-br-Hopf-3})} \\
    & \cong (-) \otimes (H \barotimes (I^* \otimes \alpha_H \otimes \alpha_{\mathcal{V}}))^{\mathrm{co}H}
    & & \text{(by Lemma~\ref{lem:Hopf-bim-dual})} \\
    & \cong (-) \otimes I^* \otimes \alpha_H \otimes \alpha_{\mathcal{V}}.
  \end{align*}
  Now the result is obtained by evaluating the both sides of \eqref{eq:mod-obj-br-Hopf-1} at $\unitobj$.
\end{proof}

\subsection{Proof of Theorem~\ref{thm:FMS-braided}}

We now prove Theorem~\ref{thm:FMS-braided}. The assumption is that $\mathcal{V}$ is a braided finite tensor category over $k$ and $i_{A/B}: B \to A$ is an extension of Hopf algebras in $\mathcal{V}$.

\begin{proof}[Proof of Theorem~\ref{thm:FMS-braided}]
  Let $F: \mathcal{V}_A \to \mathcal{V}_B$ be the restriction functor. It is sufficient to show that Theorem~\ref{thm:mod-obj-3} is applicable to $F$. For $M \in \mathcal{V}_B$, we denote its underlying object by $M_0$ for clarity. As before, we regard $\mathcal{V} \subset \mathcal{V}_B$ (and thus $M_0 \in \mathcal{V}_B$). For $X \in \mathcal{V}_B$, we consider the morphism
  \begin{equation*}
    X_0 \otimes I_0 \otimes ({}_B A)^*
    \xrightarrow{\ \id_X \otimes \id_I \otimes i_{A/B}^* \ }
    X_0 \otimes I_0 \otimes ({}_B B)^*
    \xrightarrow{\ \id_X \otimes \phi_B^{-1} \ }
    X_0 \otimes B
    \xrightarrow{\ \triangleleft_X \ } X,
  \end{equation*}
  where $I = \mathrm{Int}(B)$ and $\phi_B$ is the isomorphism given by \eqref{eq:br-Hopf-pairing} with $H = B$. This is an epimorphism of right $B$-modules. Obviously, $X_0 \otimes I_0 \otimes ({}_B A)^*$ is a restriction of an $A$-module. Thus, by Remark~\ref{rem:NZ-thm}, we can apply Theorem~\ref{thm:mod-obj-3} to $F$.
\end{proof}

\section{Radford-Majid bosonization}
\label{sec:boson}

\subsection{Main results of this section}

Given a Hopf algebra $H$ over $k$ and a Hopf algebra $\mathbf{B}$ in the Yetter-Drinfeld category ${}^H_H \mathcal{YD}$, one can construct an ordinary Hopf algebra $\mathbf{B} \# H$, called the {\em Radford-Majid bosonization}. This construction is important in the Hopf algebra theory, especially in the classification program of pointed Hopf algebras.

There is a category-theoretical counterpart of the Radford-Majid bosonization: Let $\mathcal{C}$ be a monoidal category, and let $\mathbf{B} = (B, \tau)$ be a Hopf algebra in $\mathcal{Z}(\mathcal{C})$. Since $\mathcal{Z}(\mathcal{C})$ acts on $\mathcal{C}$ via the forgetful functor, the category ${}_{\mathbf{B}}\mathcal{C}$ of left $\mathbf{B}$-modules in $\mathcal{C}$ is defined. The tensor product of $M, N \in {}_{\mathbf{B}}\mathcal{C}$ is a left $\mathbf{B}$-module in $\mathcal{C}$ by
\begin{equation*}
  \triangleright_{M \otimes N} =
  (\triangleright_M \otimes \triangleright_N) \circ (\id_{B} \otimes \tau_{M} \otimes \id_N)
  \circ (\Delta \otimes \id_M \otimes \id_N),
\end{equation*}
where $\triangleright_M$ and $\triangleright_N$ are the actions of $\mathbf{B}$ on $M$ and $N$, respectively. The category ${}_{\mathbf{B}}\mathcal{C}$ is a monoidal category with respect to this operation. It is well-known that, if $\mathcal{C}$ is the category of left $H$-modules, then $\mathcal{Z}(\mathcal{C})$ can be identified with ${}^H_H \mathcal{YD}$, and the category ${}_{\mathbf{B}}\mathcal{C}$ is canonically isomorphic to the category of left $\mathbf{B} \# H$-modules as a monoidal category.

We now consider the case where $\mathcal{C}$ is a finite tensor category. Then ${}_{\mathbf{B}}\mathcal{C}$ is also a finite tensor category. In view of applications to generalizations of the bosonization construction, we are interested in determining the modular object of ${}_{\mathbf{B}}\mathcal{C}$. With the notation used in Section~\ref{sec:braided-Hopf}, our main result is now stated as follows:

\begin{theorem}
  \label{thm:mod-obj-boson}
  If the forgetful functor $U: \mathcal{Z}(\mathcal{C}) \to \mathcal{C}$ satisfies the equivalent conditions of Lemma~\ref{lem:rel-mod-def}, then the modular object of $\mathcal{A} = {}_{\mathbf{B}}\mathcal{C}$ is given by
  \begin{equation*}
    \alpha_{\mathcal{A}} = U(\mathrm{Int}(\mathbf{B}))^* \otimes \alpha_{\mathcal{C}} \otimes U(\overline{\alpha}_{\mathbf{B}}).
  \end{equation*}
\end{theorem}

We note that the assumption on $U$ is satisfied if, for example, the base field $k$ is perfect (see Example~\ref{ex:rel-mod-cent}).
This result resembles Corollary~\ref{cor:mod-obj-br-Hopf}.
However, the techniques we have used in Section~\ref{sec:braided-Hopf} do not seem to be applicable.
Thus we rather derive this theorem by using Corollary~\ref{cor:mod-obj-br-Hopf} and some fundamental properties of the relative modular object.

Before we give a proof of this theorem, we give some applications. Let $\mathcal{C}$ and $\mathbf{B}$ be as in the above theorem. The first one reduces to a result on the unimodularity of a finite-\hspace{0pt}dimensional Hopf algebra obtained by the bosonization.

\begin{corollary}
  \label{cor:boson-1}
  ${}_{\mathbf{B}}\mathcal{C}$ is unimodular if and only if $\alpha_{\mathcal{C}} \cong U(\mathrm{Int}(\mathbf{B}))$ and $\alpha_{\mathbf{B}} = \varepsilon$.
\end{corollary}

The second corollary below closely relates to \cite[Theorem 5.6]{MR1401518} that describes the Frobenius property of an extension of Hopf algebras obtained by the bosonization.

\begin{corollary}
  \label{cor:boson-2}
  For an extension $i: \mathbf{B} \to \mathbf{A}$ of Hopf algebras in $\mathcal{Z}(\mathcal{C})$, the following assertions are equivalent:
  \begin{enumerate}
  \item The restriction functor $\mathrm{res}^{\mathbf{A}}_{\mathbf{B}}: {}_{\mathbf{A}}\mathcal{C} \to {}_{\mathbf{B}}\mathcal{C}$ is a Frobenius functor.
  \item $U(\mathrm{Int}(\mathbf{A})) \cong U(\mathrm{Int}(\mathbf{B}))$ and $\alpha_{\mathbf{A}} \circ i = \alpha_{\mathbf{B}}$.
  \end{enumerate}
\end{corollary}
\begin{proof}
  One can check that $\mathrm{res}^{\mathbf{A}}_{\mathbf{B}}$ satisfies the equivalent conditions of Lemma~\ref{lem:rel-mod-def} in the same way as Theorem~\ref{thm:FMS-braided}. The claim follows from Theorems~\ref{thm:mod-obj-3} and~\ref{thm:mod-obj-boson}.
\end{proof}

\subsection{Proof of Theorem~\ref{thm:mod-obj-boson}}

For a monoidal category $\mathcal{V}$, we denote by $\mathcal{V}\mbox{-\underline{Mod}}$ the 2-category of essentially small left $\mathcal{V}$-module categories, lax $\mathcal{V}$-module functors and morphisms between them.
Let $A$ be an algebra in $\mathcal{V}$. In this paper, we have repeatedly used the fact that a lax $\mathcal{V}$-module functor $F: \mathcal{M} \to \mathcal{N}$ induces a functor ${}_A F: {}_A \mathcal{M} \to {}_A \mathcal{N}$ between the categories of $A$-modules. This fact is a part of the general fact that there is a 2-functor
\begin{equation*}
  {}_A (-): \mathcal{V}\mbox{-\underline{Mod}}
  \to \text{(the 2-category of categories)}, \quad \mathcal{M} \mapsto {}_A \mathcal{M}
\end{equation*}
given by taking the category of left $A$-modules. An important consequence is that an adjunction $F \dashv G$ in $\mathcal{V}\mbox{-\underline{Mod}}$ yields an ordinary adjunction ${}_A F \dashv {}_A G$.

Now let $\mathcal{C}$ be a finite tensor category over a field $k$. To clarify our argument, we consider a more general setting than Theorem~\ref{thm:mod-obj-boson}. Let $\mathcal{V}$ be a braided finite tensor category, and let $F: \mathcal{V} \to \mathcal{C}$ be a tensor functor satisfying the equivalent conditions of Lemma~\ref{lem:rel-mod-def} so that it admits the relative modular object. We moreover assume that $F$ is {\em central} in the sense that there is a braided monoidal functor $\widetilde{F}: \mathcal{V} \to \mathcal{Z}(\mathcal{C})$ such that $U \circ \widetilde{F} = F$.
Let $H$ be a Hopf algebra in $\mathcal{V}$. Since $\mathcal{V}$ acts on $\mathcal{C}$ via $F$, the category ${}_H \mathcal{C}$ of $H$-modules in $\mathcal{C}$ is defined. By the assumption that $F$ is central, ${}_H \mathcal{C}$ is in fact a finite tensor category and the functor ${}_H F: {}_H \mathcal{V} \to {}_H C$ induced by $F$ is a tensor functor.

As in Section~\ref{sec:braided-Hopf}, we regard $\mathcal{C}$ as a full subcategory of ${}_H \mathcal{C}$ by endowing an object of $\mathcal{C}$ with the trivial action of $H$.
For simplicity, we set $\mathcal{A} = {}_H \mathcal{C}$ and $\mathcal{B} = {}_H \mathcal{V}$.
Then there is the following relation between the modular object of $\mathcal{A}$ and $\mathcal{B}$.

\begin{lemma}
  \label{lem:boson-1}
  $\chi_F \otimes \alpha_{\mathcal{A}} \cong {}_H F(\alpha_{\mathcal{B}})$.
\end{lemma}
\begin{proof}
  Let $E$ be a double-left adjoint ($=$ a left adjoint of a left adjoint) of $F$. By the proof of Lemma~\ref{lem:rel-mod-def}, $E \cong F(-) \otimes \chi_F$ as left $\mathcal{V}$-module functors. By the argument at the beginning of this subsection, ${}_H E$ is a double-left adjoint of ${}_H F$. Thus, by Lemma~\ref{lem:rel-mod-def}, the relative modular object of ${}_H F$ is ${}_H E(\unitobj)$, that is the object $\chi_F \in \mathcal{C}$ endowed with the trivial $H$-action. Now the claim is proved by applying Theorem~\ref{thm:mod-obj-3} to the tensor functor ${}_H F$.
\end{proof}

\begin{proof}[Proof of Theorem~\ref{thm:mod-obj-boson}]
  The assumption is that $\mathcal{C}$ is a finite tensor category such that the forgetful functor $U: \mathcal{Z}(\mathcal{C}) \to \mathcal{C}$ admits the relative modular object and $\mathbf{B}$ is a Hopf algebra in $\mathcal{Z}(\mathcal{C})$. For simplicity, we set $\mathcal{A} = {}_{\mathbf{B}}\mathcal{C}$, $\mathcal{B} = {}_{\mathbf{B}}\mathcal{Z}(\mathcal{C})$, and $\mathbf{I} = \mathrm{Int}(\mathbf{B})$ and write ${}_{\mathbf{B}}U$ simply as $U$. Then we have
  \begin{align*}
    \alpha_{\mathcal{A}}
    & \cong \chi_{U}^* \otimes U(\alpha_{\mathcal{B}})
    & & \text{(by Lemma~\ref{lem:boson-1} with $F = U$)} \\
    & \cong \chi_{U}^* \otimes U(\mathbf{I}^* \otimes \alpha_{\mathcal{Z}(\mathcal{C})} \otimes \overline{\alpha}_{\mathbf{B}})
    & & \text{(by Corollary~\ref{cor:mod-obj-br-Hopf})} \\
    & = \chi_{U}^* \otimes U(\mathbf{I})^* \otimes U(\alpha_{\mathcal{Z}(\mathcal{C})}) \otimes U(\overline{\alpha}_{\mathbf{B}}) \\
    & \cong U(\mathbf{I})^* \otimes \chi_U^* \otimes U(\alpha_{\mathcal{Z}(\mathcal{C})}) \otimes U(\overline{\alpha}_{\mathbf{B}})
    & & \text{(since $\mathbf{I} \in \mathcal{Z}(\mathcal{C})$)} \\
    & \cong U(\mathbf{I})^* \otimes \alpha_{\mathcal{C}} \otimes U(\overline{\alpha}_{\mathbf{B}})
    & & \text{(by Theorem~\ref{thm:mod-obj-3})}. \qedhere
  \end{align*}
\end{proof}

\def\cprime{$'$}

\end{document}